\documentclass[11pt]{amsart}
\usepackage{amscd,amsfonts,amssymb,amsmath,amsthm,latexsym}

\usepackage[dvips]{graphicx}
\usepackage{graphics}

\usepackage{geometry}
\usepackage{mathtools}
\usepackage{enumerate}
\usepackage{float}
\usepackage{color}
\usepackage{url}
\usepackage{array}
\usepackage{xcolor}
\usepackage{tikz}
\usetikzlibrary{calc}
\usetikzlibrary{decorations.markings}

\usepackage[all]{xy}
\xyoption{curve}
\xyoption{import}
\xyoption{arc}
\xyoption{ps}

\DeclareMathAlphabet{\mathpzc}{OT1}{pzc}{m}{it}

\numberwithin{equation}{section}

\theoremstyle{plain}
    \newtheorem{thm}{Theorem}[section]
    \newtheorem{lemma}[thm]{Lemma}
    \newtheorem{coro}[thm]{Corollary}
    \newtheorem{prop}[thm]{Proposition}

\theoremstyle{definition}
    \newtheorem{defi}[thm]{Definition}
    \newtheorem{ex}[thm]{Example}
    \newtheorem{remark}[thm]{Remark}
\theoremstyle{remark}
    
    \newtheorem{claim}{Claim}

\newcommand{\suchthat}{\ | \ }

\newcommand{\field}{K}

\newcommand{\marked}{\mathbb{M}}

\newcommand{\surf}{(\Sigma,\marked)}

\newcommand{\RA}[1]{\field\langle\hspace{-0.05cm}\langle #1\rangle\hspace{-0.05cm}\rangle}

\newcommand{\myid}{1\hspace{-0.15cm}1}

\newcommand{\Qtau}{Q(\tau)}

\newcommand{\Staux}{S(\tau,\mathbf{x})}

\newcommand{\Qsigma}{Q(\sigma)}

\newcommand{\Stauxn}{S(\tau,x,n)}
\newcommand{\Ssigmaxn}{S(\sigma,x,n)}

\newdir{((}{{}*!/-2\jot/\dir^{(}}
\newdir{(((}{{}*!/-2.8\jot/\dir^{(}}

\newcommand{\Ttau}{T(\tau)}
\newcommand{\Tsigma}{T(\sigma)}
\newcommand{\short}{\operatorname{short}}
\newcommand{\depth}{\operatorname{depth}}
\newcommand{\xx}{\mathbf{x}}
\newcommand{\val}{\operatorname{val}}
\newcommand{\Gsymbol}{\mathcal{G}}

\title[QPs associated to triangulations of closed surfaces with at most two punctures]{Quivers with potentials associated to triangulations of closed surfaces with at most two punctures}


\author{Jan Geuenich}
\address{Jan Geuenich\newline
Fakult\"{a}t f\"{u}r Mathematik\newline 
Universit\"{a}t Bielefeld\newline 
Germany}
\email{jgeuenich@math.uni-bielefeld.de}

\author{Daniel Labardini-Fragoso}
\address{Daniel Labardini-Fragoso\newline
Instituto de Matem\'aticas\newline
Universidad Nacional Aut{\'o}noma de M{\'e}xico\newline
Mexico}
\email{labardini@im.unam.mx}

\author{Jos\'e Luis Miranda-Olvera}
\address{Jos\'e Luis Miranda-Olvera\newline
Department of Mathematical Sciences\newline
Carnegie Mellon University\newline
USA}
\email{joseluismiranda@cmu.edu}

\date{\today}

\begin{document}

  \begin{abstract}
We tackle the classification problem of non-degenerate potentials for quivers arising from triangulations of surfaces in the cases left open by Geiss-Labardini-Schr\"{o}er. Namely, for once-punctured closed surfaces of positive genus, we show that the quiver of any triangulation admits infinitely many non-degenerate potentials that are pairwise not weakly right-equivalent; we do so by showing that the potentials obtained by adding the 3-cycles coming from triangles and a fixed power of the cycle surrounding the puncture are well behaved under flips and QP-mutations. For twice-punctured closed surfaces of positive genus, we prove that the quiver of any triangulation admits exactly one non-degenerate potential up to weak right-equivalence, thus confirming the veracity of a conjecture of the aforementioned authors.
\end{abstract}

  \maketitle

  \tableofcontents

  \section{Introduction}

Albeit technical in nature, the problem of classifying 
all non-degenerate potentials on a given 2-acyclic quiver is relevant in different interesting, seemingly unrelated, contexts. In cluster algebra theory, having only one weak right-equivalence class means, very roughly speaking, that Derksen-Weyman-Zelevinsky's representation-theoretic approach to the corresponding cluster algebra can be performed in essentially only one way.

The classification problem of non-degenerate potentials plays a role also in algebraic geometry and in symplectic geometry (more precisely, in the subjects of Bridgeland stability conditions and Fukaya categories). In \cite[Theorem 9.9]{Bridgeland-Smith}, the uniqueness of non-degenerate potentials on the quivers arising from positive genus closed surfaces with at least three punctures is used by Tom Bridgeland and Ivan Smith to prove that there is a short exact sequence
$$
\xymatrix{
1 \ar[r] & \mathpzc{Sph}_\triangle(\mathcal{D}\surf) \ar[r] &\mathpzc{Aut}_{\triangle}(\mathcal{D}\surf) \ar[r] & \operatorname{MCG}^{\pm}\surf \ar[r] & 1,
}
$$
where $\mathcal{D}\surf$ is the 3-Calabi-Yau triangulated category associated to $\surf$, defined as the full subcategory that the dg-modules with finite-dimensional cohomology determine inside the derived category of the Ginzburg dg-algebra of the quiver with potential of \textbf{any}\footnote{That $\mathcal{D}\surf$ is independent of the tagged triangulation used follows after combining results of Keller-Yang \cite{Keller-Yang} and Labardini \cite{Labardini4}, see \cite[Section 5]{Labardini-survey}.} tagged triangulation of $\surf$, the group $\mathpzc{Aut}_{\triangle}(\mathcal{D}\surf)$ is the quotient of the group of auto-equivalences of $\mathcal{D}\surf$ that preserve the distinguished connected component $\operatorname{Tilt}_{\triangle}(\mathcal{D}\surf)$ by the subgroup of auto-equivalences that act trivially on $\operatorname{Tilt}_{\triangle}(\mathcal{D}\surf)$, $\mathpzc{Sph}_\triangle(\mathcal{D}\surf)$ is the subgroup of $\mathpzc{Aut}_{\triangle}(\mathcal{D}\surf)$ generated by (the quotient images of) the twist functors at the simple objects of an heart $\mathcal{A}\in \operatorname{Tilt}_{\triangle}(\mathcal{D}\surf)$,  and $\operatorname{MCG}^{\pm}\surf=\operatorname{MCG}\surf\ltimes\mathbb{Z}_2^{\marked}$ is the signed mapping class group.

In \cite[Theorem 1.1]{Smith}, Ivan Smith shows that if $\surf$ is a positive genus closed surface with at least three punctures (i.e., $|\marked|\geq 3$), then there is a linear fully faithful embedding of the 3-Calabi-Yau triangulated category $\mathcal{D}\surf$ into a Fukaya category of a 3-fold that fibers over $\Sigma$ (with poles of a quadratic differential removed from $\Sigma$). He explains that the reason behind the hypothesis $|\mathbb{M}| \geq 3$ in \cite[Theorem 1.1]{Smith} arises from the fact that for positive genus closed surfaces with at least three punctures, the quiver of any triangulation has exactly one non-degenerate potential up to weak right-equivalence (a fact shown by Geiss-Labardini-Schr\"oer \cite{GLFS}).  See \cite[Sections 1.3 and~2.2]{Smith}. 

Together with results from his work \cite{Bridgeland-Smith} with Bridgeland, the embeddings from the previous paragraph allow Smith to obtain non-trivial computations of spaces of stability conditions on Fukaya categories of symplectic six-manifolds.

In this paper we prove the following:

\begin{thm}\label{main-thm-intro}\begin{enumerate}\item\label{item:main-thm-one-puncture} For once-punctured closed surfaces of positive genus, the quiver of any triangulation admits infinitely many non-degenerate potentials that are pairwise not weakly right-equivalent, provided the underlying field has characteristic zero.
\item\label{item:main-thm-two-punctures} For twice-punctured closed surfaces of positive genus, the quiver of any triangulation admits exactly one non-degenerate potential up to weak right-equivalence, provided the underlying field is algebraically closed.
\end{enumerate}
\end{thm}

Let $\surf$ be a once-punctured closed surface, $n$ a positive integer and $x\in\field$ any scalar. For a triangulation $\tau$ of $\surf$ let $\Stauxn$ be the potential obtained by adding the 3-cycles of $Q(\tau)$ arising from triangles of $\tau$ and the $x$-multiple of the $n^{\operatorname{th}}$ power of the cycle of $Q(\tau)$ that runs around the puncture of $\surf$. The following result of independent interest plays a central role in our proof of part \eqref{item:main-thm-one-puncture} of Theorem \ref{main-thm-intro}:

\begin{thm} Let $\surf$ be a once-punctured closed surface. If $\tau$ and $\sigma$ are triangulations of $\surf$ related by the flip of an arc $k\in\tau$, then the quivers with potential $(Q(\tau),\Stauxn)$ and $(Q(\sigma),\Ssigmaxn)$ are related by the mutation of quivers with potential $\mu_k$.
\end{thm}

That the quivers associated to triangulations of once-punctured closed surfaces of positive genus admit more than one weak-right-equivalence class of non-degenerate potentials has always been expected, since the works \cite{DWZ1} of Derksen-Weyman-Zelevinsky and \cite{Labardini1,Labardini2} of Labardini exhibit non-degenerate potentials for the Markov quiver\footnote{The Markov quiver arises as the quiver associated to any triangulation of the once-punctured torus.} that are not weakly right-equivalent. 

In \cite[Theorem 8.4]{GLFS}, Geiss-Labardini-Schr\"{o}er proved that every quiver associated to some triangulation of a positive-genus closed surface with at least three punctures admits exactly one weak-right-equivalence class of non-degenerate potentials, and conjectured that the same result holds in the case of two punctures. The reason why their proof fails for twice-punctured closed surface is that these do not admit triangulations all of whose arcs connect distinct punctures. The fact that such triangulations do exist for closed surfaces with at least three punctures plays an essential role in the proof of \cite[Theorem 8.4]{GLFS}.

The structure of the paper is straightforward: in Section \ref{sec:preliminaries} we prove a few facts (some of them quite technical) about the form of cycles and non-degenerate potentials for quivers arising from combinatorially nice triangulations of surfaces with empty boundary (see conditions \eqref{eq:valency-at-least-4} and \eqref{eq:no-double-arrows}). Section \ref{sec:once-punctured-surfaces} is devoted to proving part \eqref{item:main-thm-one-puncture} of Theorem \ref{main-thm-intro}, whereas Section \ref{sec:twice-punctured-surfaces} is devoted to showing part \eqref{item:main-thm-two-punctures}.
  \section{Preliminaries}

\label{sec:preliminaries}

Let $\field$ be any field. For a quiver $Q$, the \emph{vertex span} is the $\field$-algebra $R$ defined as the $\field$-vector space with basis $\{e_j\suchthat j\in Q_0\}$, with multiplication defined as the $\field$-bilinear extension of the rule
$$
e_ie_j:=\delta_{i,j}e_j \hspace{1cm} \text{for all $i,j\in Q_0$},
$$
where $\delta_{i,j}\in\field$ is the \emph{Kronecker delta} of $i$ and $j$. Thus, $R$ is (a $\field$-algebra isomorphic to) $K^{Q_0}$ with both sum and multiplication defined componentwise.
The \emph{complete path algebra} of $Q$ is the $\field$-vector space
$$
\RA{Q} := \prod_{\ell\in\mathbb{Z}_{\geq 0}} A^{(\ell)},
$$
where $A^{(0)}:=R$, and for $\ell>0$, $A^{(\ell)}$ is the $\field$-vector space with basis all the paths of length $\ell$ on $Q$. The multiplication of $\RA{Q}$ is defined in terms of the concatenation of paths.

The vertex span $R$ is obviously a subring of $\RA{Q}$ (actually, a $\field$-subalgebra), but it is often not a central subring. Despite this, any ring automorphism $\varphi:\RA{Q}\rightarrow\RA{Q}$ such that $\varphi|_{R}=\myid_{R}$ will be said to be an \emph{$R$-algebra automorphism} of $\RA{Q}$.

\begin{defi} Let $Q$ be a quiver and $S,W\in \RA{Q}$ be potentials on $Q$. We will say that:
\begin{enumerate}
\item two cycles $a_1\cdots a_\ell$ and $b_1\cdots b_m$ on $Q$ are \emph{rotationally equivalent} if $a_1\cdots a_\ell=b_1\cdots b_m$ or $a_1\cdots a_\ell = b_k\cdots b_m b_1\cdots b_{k-1}$ for some $k\in\{2,\ldots,m\}$;
\item $S$ and $W$ are \emph{rotationally disjoint} if no cycle appearing in $S$ is rotationally equivalent to a cycle appearing in $W$;
\item $S$ and $W$ are \emph{cyclically equivalent} if, with respect to the $\mathfrak{m}$-adic topology of $\RA{Q}$, the element $S-W$ belongs to the topological closure of the vector subspace of $\RA{Q}$ spanned by all elements of the form $a_1\cdots a_\ell-a_2\cdots a_\ell a_1$ with $a_1\cdots a_\ell$ running through the set of all cycles on $Q$; notation: $S\sim_{\operatorname{cyc}}W$;
\item $S$ and $W$ are \emph{right-equivalent} if there exists a \emph{right equivalence} from $S$ to $W$, i. e., an $R$-algebra automorphism $\varphi:\RA{Q}\rightarrow\RA{Q}$ that acts as the identity on the set of idempotents $\{e_j\suchthat j\in Q_0\}$ and satisfies $\varphi(S)\sim_{\operatorname{cyc}}W$; notation: $S\sim_{\operatorname{r.e.}}W$;
\item $S$ and $W$ are \emph{weakly right-equivalent} if $S$ and $\lambda W$ are right-equivalent for some non-zero scalar $\lambda\in\field$.
\end{enumerate}
\end{defi}

Throughout the paper, $\surf$ will be a punctured closed surface of positive genus. That is, $\Sigma$ will be a compact, connected, oriented two-dimensional real differentiable manifold with positive genus and empty boundary, and $\mathbb{M}$ will be a non-empty finite subset of $\Sigma$.

It is very easy to show that there exists at least one triangulation $\tau$ of $(\Sigma,\mathbb{M})$ such that
\begin{eqnarray}
\label{eq:valency-at-least-4}
&&\text{Every puncture has valency at least $4$ with respect to $\tau$;}\\
\label{eq:no-double-arrows}
&&\text{for any two arcs $i$ and $j$ of $\tau$, the quiver $\Qtau$ has at most one arrow from $j$ to $i$.}
\end{eqnarray}
Throughout the paper, we will permanently suppose that $\tau$ satisfies \eqref{eq:valency-at-least-4} and \eqref{eq:no-double-arrows}.

Following Ladkani \cite{Ladkani} we define two maps $f,g:\Qtau_1\rightarrow\Qtau_1$ as follows. Each triangle $\triangle$ of $\tau$ gives rise to an oriented 3-cycle $\alpha_\triangle\beta_\triangle\gamma_\triangle$ on $\Qtau$. We set $f(\alpha_\triangle)=\gamma_\triangle$, $f(\beta_\triangle)=\alpha_\triangle$ and $f(\gamma_\triangle)=\beta_\triangle$. Now, given any arrow $\alpha$ of $\Qtau$, the quiver $\Qtau$ has exactly two arrows starting at the terminal vertex of $\alpha$. One of these two arrows is $f(\alpha)$. We define $g(\alpha)$ to be the other arrow.

Note that the map $f$ (resp. $g$) splits the arrow set of $\Qtau$ into $f$-orbits (resp. $g$-orbits). The set of $f$-orbits is in one-to-one correspondence with the set of triangles of $\tau$. All $f$-orbits have exactly three elements. The set of $g$-orbits is in one-to-one correspondence with the set of punctures of $(\Sigma,\mathbb{M})$. For every arrow $\alpha$ of $\Qtau$, we denote by $m_\alpha$ the size of the $g$-orbit of $\alpha$ ($m_\alpha \geq 4$ by \eqref{eq:valency-at-least-4}). Note that $g^{m_\alpha-1}(\alpha)g^{m_\alpha-2}(\alpha)\cdots g(\alpha)\alpha$ is a cycle surrounding the puncture $p$ corresponding to the $g$-orbit of $\alpha$, we denote this cycle as $\Gsymbol (\alpha)$ or $\Gsymbol(p)$. Whereas, for every arrow $\beta$ of $\Qtau$ and any non-negative integer $r$, we use the notation $G(r,\beta)$ to denote the path $g^{r-1}(\beta)g^{r-2}(\beta)\cdots g(\beta)\beta$. Similarly, we use the notation $F(r,\beta)$ to denote the path $f^{r-1}(\beta)f^{r-2}(\beta)\cdots f(\beta)\beta$.

Let $\mathbf{x}=(x_p)_{p\in\mathbb{M}}$ be a choice of a non-zero scalar $x_p\in\field$ for each puncture $p\in\mathbb{M}$. For ideal triangulations which satisfy \eqref{eq:valency-at-least-4} and \eqref{eq:no-double-arrows} the potential $\Staux$ defined by the second author \cite{Labardini4} takes a simple form, namely,
$$
\Staux = \Ttau + \sum_{p\in \mathbb{P}}x_p\Gsymbol(p),
$$
with $ \Ttau \sim_{\operatorname{cyc}} \sum_{\alpha\in \Gamma}(f^2(\alpha)f(\alpha)\alpha)$ for any fixed set $\Gamma$ containing exactly one arrow from each triangle of $\tau$.

\begin{lemma}[Types of cycles]\label{lemma:f-g-fg-potentials} Let $\surf$ be a punctured surface with empty boundary, and let $\tau$ be a triangulation of $\surf$ that satisfies \eqref{eq:valency-at-least-4} and \eqref{eq:no-double-arrows}. Then every cycle in $Q(\tau)$ is rotationally equivalent to a cycle of one of the following types:
\begin{itemize}
\item[($f$-cycles)] $(f^2(\alpha)f(\alpha)\alpha)^n$ for some $n\geq 1$;
\item[($g$-cycles)] $(g^{m_\beta-1}(\beta)g^{m_\beta-2}(\beta)\cdots g(\beta)\beta)^n$ for some $n\geq 1$;
\item[($fg$-cycles)] $f^2(a)f(a)\lambda$ for some arrow $a$ and some path $\lambda$, such that $\lambda=g^{-1}f(a)\lambda'$ with $\lambda'$ of positive length.
\end{itemize}

\begin{proof}
Let $\xi = \alpha_1\cdots\alpha_r$ be any cycle on $\Qtau$. Denote $\alpha_{r+1}=\alpha_1$, and notice that for every $\ell =  1, \cdots , r$, we have either $\alpha_\ell = f(\alpha_{\ell +1})$ or $\alpha_\ell = g(\alpha_{\ell +1})$. Let $\mathbf{s}_\xi$ be the length-$r$ sequence of $f$s and $g$s that has an $f$ at the $\ell^{\operatorname{th}}$ place if $\alpha_\ell = f(\alpha_{\ell +1})$ and a $g$ otherwise.

If $\mathbf{s}_\xi$ consists only of $f$s, then $\xi$ is rotationally equivalent to $(f^2(\alpha)f(\alpha)\alpha)^n$ for some arrow $\alpha$ and some $n\geq 1$. Furthemore, if $\mathbf{s}_\xi$ consists only of $g$s, then $\xi$ is rotationally equivalent to $(g^{m_\beta-1}(\beta)g^{m_\beta-2}(\beta)\cdots g(\beta)\beta)^n$ for some arrow $\beta$ and some $n\geq 1$. Therefore, if $\mathbf{s}_\xi$ involves only $f$s or only $g$s, then $\xi$ is an $f$-cycle or a $g$-cycle.

Suppose that at least one $f$ and at least one $g$ appear in $\mathbf{s}_\xi$. Rotating $\xi$ if necessary, we can assume that $\mathbf{s}_\xi$ starts with an $f$ followed by a $g$, i.e., $\mathbf{s}_\xi=(f,g,\cdots)$. This means that if we set $a:=f^{-1}(\alpha_2)$, then $\alpha_1=f^2(a)$, $\alpha_2=f(a)$ and $\alpha_3=g^{-1}f(a)$. By \eqref{eq:no-double-arrows} $a$ is the only arrow in $\Qtau_1$ such that $\alpha_1\alpha_2a$ is a cycle. Since $\alpha_3=g^{-1}f(a)\neq a$, this implies $\xi=f^2(a)f(a)g^{-1}f(a)\lambda'$ with $\lambda'$ of positive length.
\end{proof}
\end{lemma}

\begin{remark}
As in the case of cycles, every path falls within exactly one of three types of paths: $f$-paths, $g$-paths, and $fg$-paths.
\end{remark}

By Lemma \ref{lemma:f-g-fg-potentials}, up to cyclical equivalence we can write every potential $S$ in $\Qtau$ as $S=S_f+S_g+S_{fg}$, where
\begin{eqnarray*}
S_f & = & \sum_{\triangle} \sum_{n=1}^{\infty } z_{\triangle,n}(f^2(\alpha_\triangle)f(\alpha_\triangle)\alpha_\triangle)^n,\\
S_g & = & \sum_{p\in \mathbb{P}}\sum_{n=1}^\infty \nu_{p,n}(\Gsymbol (p))^n,\\
S_{fg} & = & \sum_{a\in \Qtau_1 }f^2(a)f(a)\omega_a,
\end{eqnarray*}
with each $z_{\triangle,n},\nu_{p,n} \in K$, and $\omega_a$ a possibly infinite linear combination of paths of the form $g^{-1}f(a)\lambda'$ for each $a \in Q(\tau)$.

\begin{lemma}\label{lemma:non-degenerate-potentials-contains-Ttau} Let $\surf$ be a punctured surface with empty boundary, and let $\tau$ be a triangulation of $\surf$ that satisfies \eqref{eq:valency-at-least-4} and \eqref{eq:no-double-arrows}. Every non-degenerate potential $S$ on $\Qtau$ is right-equivalent to a potential of the form $\Ttau+U$ for some $U$ rotationally disjoint from $\Ttau$.
\end{lemma}

\begin{proof}
By \eqref{eq:no-double-arrows} the hypotheses of \cite[Corollary 2.5]{GLFS} are satisfied, so if $S$ is a non-degenerate potential, then every $f$-cycle $f^2(\alpha)f(\alpha)\alpha$ appears in $S$. So,
\begin{align*}
S&\sim_{\operatorname{cyc}}\sum_{\triangle} z_{\triangle,1}f^2(\alpha_\triangle)f(\alpha_\triangle)\alpha_\triangle + U',
\end{align*}
with all $z_{\triangle,1}\neq 0$ and $U'$ rotationally disjoint from $\Ttau$.

We define an $R$-algebra automorphism $\varphi:\RA{\Qtau}\rightarrow\RA{\Qtau}$ by means of the rule
\begin{equation*}
\varphi(\alpha_\triangle) = \frac{1}{z_{\triangle,1}}\alpha_\triangle.
\end{equation*}
We see that $\varphi(S)\sim_{\operatorname{cyc}}\Ttau + U$, for some potential $U$ rotationally disjoint from $\Ttau$.
\end{proof}

\begin{lemma}[Replacing $f$-potentials and $fg$-potentials by longer ones]\label{lemma:Af-Afg-inequalities} Let $\surf$ be a punctured surface with empty boundary, and let $\tau$ be a triangulation of $\surf$ that satisfies \eqref{eq:valency-at-least-4} and \eqref{eq:no-double-arrows}. Let $\phi$ be one of the symbols $f$ and $fg$, and let $\nu$ be the other symbol, so that $\{\phi,\nu\}=\{f,fg\}$ as sets of symbols. If $W,A\in\RA{\Qtau}$ are potentials rotationally disjoint from $\Ttau$, and if $A_\phi\neq 0$, then there exists a potential $B\in\RA{\Qtau}$ which is rotationally disjoint from $\Ttau$ and satisfies the following four conditions:
\begin{eqnarray*}
\short(B_\phi)&>&\short(A_\phi);\\
\short(B_g)&\geq&\min(\short(A_g),\short(A_\phi)+1);\\
\short(B_{\nu})&\geq&\min(\short(A_{\nu}),\short(A_\phi)+1);\\
(\Qtau,\Ttau+W+A)&\sim_{\operatorname{r.e.}}&(\Qtau,\Ttau+W+B).
\end{eqnarray*}
\end{lemma}

\begin{proof} Let us deal with the case $\phi=f$. Write
\begin{equation*}
A_f=\sum_{\triangle}\sum_{n\geq \frac{\short(A_f)}{3}}z_{\triangle,n}\left(f^2(\alpha_\triangle)f(\alpha_\triangle)\alpha_\triangle\right)^n
\end{equation*}
and define an $R$-algebra homomorphism $\varphi:\RA{\Qtau}\rightarrow\RA{\Qtau}$ by means of the rule
\begin{equation*}
\varphi(\alpha_\triangle) = \alpha_\triangle-\sum_{n\geq \frac{\short(A_f)}{3}}z_{\triangle,n}\alpha_\triangle\left(f^2(\alpha_\triangle)f(\alpha_\triangle)\alpha_\triangle\right)^{n-1}.
\end{equation*}
Then $\varphi$ is a unitriangular automorphism of depth $\short(A_f)-3$, and
\begin{equation*}
\varphi(\Ttau+W+A)=\Ttau-A_f+W+A+(\varphi(W+A)-(W+A)).
\end{equation*}
Consequently, if we set $B=A_g+A_{fg}+(\varphi(W+A)-(W+A))$, then:
\begin{itemize}
\item $\varphi(\Ttau+W+A)=\Ttau+W+B$;
\item $\short(\varphi(W+A)-(W+A))\geq \depth(\varphi)+\short(W+A)\geq \short(A_f)-3+4=\short(A_f)+1$;
\item $\short(B_f)=\short((\varphi(W+A)-(W+A))_{f})\geq \short(A_f)+1$;
\item $\short(B_g)\geq\min(\short(A_g),\short((\varphi(W+A)-(W+A))_g)\geq \min(\short(A_g),\short(A_f)+1)$; and
\item $\short(B_{fg})\geq\min(\short(A_{fg}),\short((\varphi(W+A)-(W+A))_{fg}))\geq \min(\short(A_{fg}),\short(A_f)+1)$.
\end{itemize}

Now we deal with the case $\phi=fg$. Write
\begin{equation*}
A_{fg}=\sum_{a\in Q(\tau)_1}f^2(a)f(a)\omega_a,
\end{equation*}
with $\omega_a\in e_{h(a)}\RA{\Qtau}e_{t(a)}$ for each $a\in\Qtau_1$
and define an $R$-algebra homomorphism $\varphi:\RA{\Qtau}\rightarrow\RA{\Qtau}$ by means of the rule
$
\varphi(a) = a-\omega_a$
for $a\in \Qtau_1$. Then $\varphi$ is a unitriangular automorphism of depth $\short(A_{fg})-3$, and
\begin{eqnarray*}
\varphi(\Ttau+W+A)&=&
\sum_{\triangle}\left(f^2(\alpha_\triangle)-\omega_{f^2(\alpha_\triangle)}\right)\left(f(\alpha_\triangle)-\omega_{f(\alpha_\triangle)}\right)\left(\alpha_\triangle-\omega_{\alpha_\triangle}\right)\\
&&
+W+A+\left(\varphi(W+A)-(W+A)\right)\\
&\sim_{\operatorname{cyc}}&
\Ttau+W+A_f+A_g+\left(\varphi(W+A)-(W+A)\right)
\\
&&
+
\sum_{a\in\Qtau_1}f^2(a)\omega_{f(a)}\omega_a
-\sum_{\triangle}\omega_{f^2(a_\triangle)}\omega_{f(a_\triangle)}\omega_{a_\triangle}.
\end{eqnarray*}
Consequently, if we set $B=A_f+A_g+\left(\varphi(W+A)-(W+A)\right)
+\sum_{a\in\Qtau_1}f^2(a)\omega_{f(a)}\omega_a
-\sum_{\triangle}\omega_{f^2(a_\triangle)}\omega_{f(a_\triangle)}\omega_{a_\triangle}$, then:
\begin{itemize}
\item $\varphi(\Ttau+W+A)\sim_{\operatorname{cyc}}\Ttau+W+B$;
\item $\short\left(\varphi(W+A)-(W+A)\right)\geq \depth(\varphi)+\short(W+A)\geq
\short(A_{fg})-3+4=\short(A_{fg})+1$;
\item $\short\left(\sum_{a\in\Qtau_1}f^2(a)\omega_{f(a)}\omega_a\right)\geq 2\short (A_{fg})-3\geq \short (A_{fg})+4-3=\short (A_{fg})+1$ (since $\short(A_{fg})\geq 4$);
\item $\short\left(\sum_{\triangle}\omega_{f^2(a_\triangle)}\omega_{f(a_\triangle)}\omega_{a_\triangle}\right)\geq 3\short (A_{fg})-6\geq \short (A_{fg}) +8 -6\geq \short (A_{fg})+1$;
\item $\short(B_{fg})\geq\min\left(\short\left(\varphi(W+A)-(W+A)\right),\right.$ 
$\left.\short\left(\sum_{a\in\Qtau_1}f^2(a)\omega_{f(a)}\omega_a\right),\right.$\\ 
$\left.\short\left(\sum_{\triangle}\omega_{f^2(a_\triangle)}\omega_{f(a_\triangle)}\omega_{a_\triangle}\right)\right)
\geq \short(A_{fg}) +1$;
\item $\short(B_g)\geq\min\left(\short(A_g),\right.$ $\short\left(\varphi(W+A)-(W+A)\right),$ $\short\left(\sum_{a\in\Qtau_1}f^2(a)\omega_{f(a)}\omega_a\right),$\\ $\left.\short\left(\sum_{\triangle}\omega_{f^2(a_\triangle)}\omega_{f(a_\triangle)}\omega_{a_\triangle}\right)\right)\geq\min\left(\short(A_g),\short(A_{fg})+1\right)$; and
\item $\short(B_{f})\geq \min\left(\short(A_g),\right.$ $\short\left(\varphi(W+A)-(W+A)\right),$ $\short\left(\sum_{a\in\Qtau_1}f^2(a)\omega_{f(a)}\omega_a\right),$\\ $\left.\short\left(\sum_{\triangle}\omega_{f^2(a_\triangle)}\omega_{f(a_\triangle)}\omega_{a_\triangle}\right)\right) \geq \min\left(\short(A_f),\short(A_{fg})+1\right)$.
\end{itemize}

Lemma \ref{lemma:Af-Afg-inequalities} is proved.
\end{proof}

\begin{prop}[Replacing potentials by sums of powers of $g$-cycles]\label{prop:replacing-with-sums-of-powers-of-g-cycles} Let $\surf$ be a punctured surface with empty boundary, and let $\tau$ be a triangulation of $\surf$ satisfying \eqref{eq:valency-at-least-4} and \eqref{eq:no-double-arrows}. If $U,Z\in\RA{\Qtau}$ are potentials rotationally disjoint from $\Ttau$, then there exist a unitriangular automorphism $\varphi:\RA{\Qtau}\rightarrow\RA{\Qtau}$ of depth at least $\short(U)-3$ and a potential $W\in\RA{\Qtau}$ involving only positive powers of $g$-cycles, such that $\short(W)\geq\short(U)$ and $\varphi$ is a right-equivalence $(\Qtau,\Ttau+Z+U)\rightarrow(\Qtau,\Ttau+Z+W)$.
\end{prop}

\begin{proof} Set $W_0=U_g$ and $U_0=U-U_g$. We obviously have $\short(U_0),\short(W_0)\geq\short(U)$.

\begin{claim}\label{claim:seqs-Un-Wn-varphin} There exist sequences $(U_n)_{n\geq 1}$ and $(W_n)_{n\geq 1}$ of potentials on the quiver $\Qtau$, and a sequence $(\varphi_n)_{n\geq 1}$ of unitriangular automorphisms of $\RA{\Qtau}$, such that the following properties are satisfied for every $n\geq 1$:
\begin{itemize}
\item $\varphi_n$ is a right-equivalence $(\Qtau,\Ttau+Z+U_{n-1}+W_{n-1})\rightarrow(\Qtau,\Ttau+Z+U_n+W_n)$;
\item $\depth(\varphi_n)=\short(U_{n-1})-3$;
\item each of $U_n$ and $W_n$ is rotationally disjoint from $\Ttau$, $U_n$ does not involve powers of $g$-cycles and $W_n$ involves only powers of $g$-cycles;
\item $\short(W_n-W_{n-1})\geq\short(U_{n-1})+1$;
\item $\short(U_{n+1})\geq  \short(U_{n-1})+1$.
\end{itemize}
\end{claim}

\begin{proof}[Proof of Claim \ref{claim:seqs-Un-Wn-varphin}]
We shall produce the three sequences $(U_n)_{n\geq 1}$, $(W_n)_{n\geq 1}$ and $(\varphi_n)_{n\geq 1}$ recursively. Fix a positive integer $n$. If $U_{n-1}=0$, we set $U_n$ to be $U_{n-1}$, $W_n$ to be $W_{n-1}$ and $\varphi_n$ to be the identity of $\RA{\Qtau}$. Otherwise, let $\phi_{n-1},\nu_{n-1}\in\{f,fg\}$ be symbols such that $\{\phi_{n-1},\nu_{n-1}\}=\{f,fg\}$ and $\short((U_{n-1})_{\phi_{n-1}})\leq\short((U_{n-1})_{\nu_{n-1}})$. By the proof of Lemma \ref{lemma:Af-Afg-inequalities}, there exist a potential $V_n\in\RA{\Qtau}$ rotationally disjoint from $\Ttau$ and a unitriangular automorphism $\varphi_n:\RA{\Qtau}\rightarrow\RA{\Qtau}$ such that
\begin{itemize}
\item $\depth(\varphi_n)=\short(U_{n-1})-3$;
\item $\varphi_n$ is a right-equivalence $(\Qtau,\Ttau+Z+W_{n-1}+U_{n-1})\rightarrow(\Qtau,\Ttau+Z+W_{n-1}+V_n)$;
\item $\short((V_n)_{\phi_{n-1}})>\short((U_{n-1})_{\phi_{n-1}})$;
\item $\short((V_n)_g)\geq\min(\short((U_{n-1})_{g}),\short((U_{n-1})_{\phi_{n-1}})+1)=\short((U_{n-1})_{\phi_{n-1}})+1$;
\item $\short((V_n)_{\nu_{n-1}})\geq\min(\short((U_{n-1})_{\nu_{n-1}}),\short((U_{n-1})_{\phi_{n-1}})+1)$.
\end{itemize}
We set $U_n=V_n-(V_n)_g$ and $W_n=W_{n-1}+(V_n)_g$. It is clear that the first four properties stated in the claim are satisfied. For the fifth property, note that if $\phi_{n}=\phi_{n-1}$, then $\short((U_{n})_{\phi_{n}})>\short((U_{n-1})_{\phi_{n-1}})$, whereas if $\phi_{n}\neq\phi_{n-1}$, then
$$
\short((U_{n+1})_{\phi_{n}})>\short((U_{n})_{\phi_{n}})\geq\short(U_{n-1}) \ \ \ \text{and}
$$
$$
\short((U_{n+1})_{\phi_{n-1}})\geq
\min(\short((U_{n})_{\phi_{n-1}}),\short((U_{n})_{\phi_{n}})+1)\geq
$$
$$
\min(\short((U_{n-1})_{\phi_{n-1}})+1,\min(\short((U_{n-1})_{\nu_{n-1}}),\short((U_{n-1})_{\phi_{n-1}})+1)+1)
>\short((U_{n-1})_{\phi_{n-1}}).
$$
These facts, together with the observation that for each $n\geq 0$ we have $\short((U_n)_{\phi_n})=\short(U_n)$, allow us to deduce that $\short(U_{n+1})\geq \short(U_{n-1})+1$ for all $n\geq 1$.
\end{proof}

From the claim, we see that
$$\lim_{n\to\infty}\short(U_n)=\infty , ~\lim_{n\to\infty}\short(W_n-W_{n-1})=\infty\mbox{ and }\lim_{n\to\infty}\depth(\varphi_n)=\infty .$$ 
Hence, if we set $W=\lim_{n\to\infty}W_n$, then $\varphi:=\lim_{n\to\infty}\varphi_n\circ\ldots\circ\varphi_1$ is a right-equivalence $(\Qtau,\Ttau+Z+U)\rightarrow(\Qtau,\Ttau+Z+W)$. Proposition \ref{prop:replacing-with-sums-of-powers-of-g-cycles} follows.
\end{proof}

\begin{lemma}\label{lemma:very-technical-lemma} Let $\surf$ be a punctured surface with empty boundary, let $\tau$ be a triangulation of $\surf$ satisfying \eqref{eq:valency-at-least-4} and \eqref{eq:no-double-arrows}, and let $\xx=(x_p)_{p\in\marked}$ be any choice of non-zero scalars. Suppose that $m$ and $t$ are positive integers and $U,W\in\RA{\Qtau}$ are potentials rotationally disjoint from $\Staux$ that satisfy the following properties:
\begin{enumerate}
\item $\short(U)\geq m$;
\item $2\short(W)-3>m$;
\item $W=\lambda f(a)aG(t,g^{-t}(a))c$ for some non-zero scalar $\lambda\in K$, some arrow $a$ and some path $c$.
\end{enumerate}
Then there exists a unitriangular $R$-algebra automorphism $\zeta:\RA{\Qtau}\rightarrow\RA{\Qtau}$ of depth $\short(W)-3$ that serves as a right-equivalence between the QPs $(\Qtau,\Staux+U+W)$ and $(\Qtau,\Staux+U+U'+W')$ for some potentials $U',W'\in\RA{\Qtau}$ that satisfy:
\begin{enumerate}
\item\label{item:a-very-technical} $\short(U')>m$;
\item\label{item:b-very-technical} $\short(W')>\short(W)$;
\item\label{item:c-very-technical} $W'=\lambda'f(b)bG(t-1,g^{-(t-1)}(b))c'$ for some non-zero scalar $\lambda'$, some arrow $b$ and some path $c'$.
\end{enumerate}
\end{lemma}

\begin{proof}
Let $\zeta:\RA{\Qtau}\rightarrow\RA{\Qtau}$ be the $R$-algebra homomorphism given by the rule
$$
\zeta(f^{-1}(a))=f^{-1}(a)-\lambda G(t,g^{-t}(a))c.
$$
Since $\tau$ satisfies \eqref{eq:no-double-arrows}, $\short(W)-3$ is a positive integer by Lemma \ref{lemma:f-g-fg-potentials}, and hence $\zeta$ is actually a unitriangular automorphism of $\RA{\Qtau}$. The depth of $\zeta$ is obviously $\short(W)-3$.

The arrow $f^{-1}(a)$ connects two arcs of $\tau$. Let $p_{f^{-1}(a)}$ be the puncture at which these arcs are incident. Direct computation shows that
\begin{eqnarray*}
\zeta(\Staux+U+W) &\sim_{\operatorname{cyc}}&
\Staux
-W
-\lambda x_{p_{f^{-1}(a)}}G(m_{f^{-1}(a)}-1,gf^{-1}(a))G(t,g^{-t}(a))c\\
&&
+U+W+\left(\zeta(U+W)-(U+W)\right)\\
&=&
\Staux
-\lambda x_{p_{f^{-1}(a)}}G(m_{f^{-1}(a)}-2,g^2f^{-1}(a))gf^{-1}(a)g^{-1}(a)G(t-1,g^{-t}(a))c\\
&&
+U+\left(\zeta(U+W)-(U+W)\right)\\
&\sim_{\operatorname{cyc}}&
\Staux
-\lambda x_{p_{f^{-1}(a)}}gf^{-1}(a)g^{-1}(a)G(t-1,g^{-t}a)cG(m_{f^{-1}(a)}-2,g^2f^{-1}(a))\\
&&
+U+\left(\zeta(U+W)-(U+W)\right).
\end{eqnarray*}
So, the lemma follows if we remember that $gf^{-1}(a)=fg^{-1}(a)$ and set
\begin{eqnarray*}
U'&:=&\zeta(U+W)-(U+W),\\ 
\lambda'&:=&-\lambda x_{p_{f^{-1}(a)}},\\ 
b&:=&g^{-1}(a),\\ 
c'&:=&cG(m_{f^{-1}(a)}-2,g^2f^{-1}(a))\\
\text{and} \ W'&:=&\lambda'f(b)bG(t-1,g^{-(t-1)}(b))c'.
\end{eqnarray*}
Indeed, property \eqref{item:c-very-technical} is obviously satisfied, whereas the inequalities $\short(U)\geq m$, $\depth(\zeta)>0$ and $2\short(W)-3>m$ imply that
\begin{eqnarray*}
\short(U')&\geq &\min(\short(\zeta(U)-U),\short(\zeta(W)-W))\\
&\geq &\min(\depth(\zeta)+\short(U),\depth(\zeta)+\short(W))\\
&=& \min(\depth(\zeta)+\short(U),2\short(W)-3)\\
&>& m.
\end{eqnarray*}
Furthermore, we also have
$$
\short(W')=m_{f^{-1}(a)}-2+\short(W)-1>\short(W),
$$
where the inequality follows from the fact that $\tau$ satisfies \eqref{eq:valency-at-least-4}.
\end{proof}

\begin{coro}[Replacing certain cycles by sums of long $g$-cycles]\label{coro:essential} Under the same hypotheses of Lemma \ref{lemma:very-technical-lemma}, if the path $c$ is assumed to be an arrow, then there exists a unitriangular $R$-algebra automorphism $\Pi:\RA{\Qtau}\rightarrow\RA{\Qtau}$ of depth at least $\min(m-3,\short(W)-3)$ that serves as a right-equivalence between the QPs $(\Qtau,\Staux+U+W)$ and $(\Qtau,\Staux+U+\xi)$ for some potential $\xi$ that involves only positive powers of $g$-cycles  and satisfies $\short(\xi)>m$.
\end{coro}

\begin{proof} This corollary follows from an inductive use of Lemma \ref{lemma:very-technical-lemma}.
Set $U_0=U$, $W_0=W$, $a_0=a$, $c_0=c$ and $\lambda_0=\lambda$. Using Lemma \ref{lemma:very-technical-lemma}, we obtain a unitriangular automorphism $\zeta_1:\RA{\Qtau}\rightarrow\RA{\Qtau}$, potentials $Z_1,W_1\in\RA{\Qtau}$, an arrow $a_1$, a path $c_1$ and a non-zero scalar $\lambda_1$, such that:
\begin{enumerate}
\item $\depth(\zeta_1)=\short(W_0)-3$;
\item $\zeta_1$ is a right-equivalence $(\Qtau,\Staux+U_0+W_0)\rightarrow(\Qtau,\Staux+U_0+Z_1+W_1)$;
\item $\short(Z_1)>m$ and $\short(W_1)\geq \short(W_0)+1$;
\item $W_1=\lambda_1f(a_1)a_1G(t-1,g^{-(t-1)}(a_1))c_1$.
\end{enumerate}
Setting $U_1=U_0+Z_1$, we see that $U_1$, $W_1$, $a_1$, $c_1$ and $\lambda_1$ satisfy the hypotheses of Lemma \ref{lemma:very-technical-lemma} for the integers $m$ and $t-1$.

Assuming that for $i\in\{0,\ldots,t-1\}$ we have $U_i$, $W_i$, $a_i$, $c_i$ and $\lambda_i$ satisfying the hypotheses of Lemma \ref{lemma:very-technical-lemma} for the integers $m$ and $t-i$, we can produce a unitriangular automorphism $\zeta_{i+1}:\RA{\Qtau}\rightarrow\RA{\Qtau}$, potentials $Z_{i+1},W_{i+1}\in\RA{\Qtau}$, an arrow $a_{i+1}$, a path $c_{i+1}$ and a non-zero scalar $\lambda_{i+1}$, such that:
\begin{enumerate}
\item $\depth(\zeta_{i+1})=\short(W_i)-3$;
\item $\zeta_{i+1}$ is a right-equivalence $(\Qtau,\Staux+U_i+W_i)\rightarrow(\Qtau,\Staux+U_i+Z_{i+1}+W_{i+1})$;
\item $\short(Z_{i+1})>m$ and $\short(W_{i+1})\geq \short(W_i)+1$;
\item $W_{i+1}=\lambda_{i+1}f(a_{i+1})a_{i+1}G(t-i-1,g^{-(t-i-1)}(a_{i+1}))c_{i+1}$.
\end{enumerate}
Setting $U_{i+1}=U_{i}+Z_{i+1}$, we see that $U_{i+1}$, $W_{i+1}$, $a_{i+1}$, $c_{i+1}$ and $\lambda_{i+1}$ satisfy the hypotheses of Lemma \ref{lemma:very-technical-lemma} for the integers $m$ and $t-(i+1)$.

The composition $\zeta=\zeta_t\circ\zeta_{t-1}\circ\ldots\circ\zeta_1$ is a unitriangular automorphism of $\RA{\Qtau}$ that has depth at least $\short(W)-3$ and serves as a right-equivalence
$(\Qtau,\Staux+U+W)\rightarrow(\Qtau,\Staux+U_{t}+W_{t})$. Notice that $U_t=U+\sum_{i=1}^tZ_i$, that $\short\left(\sum_{i=1}^tZ_i\right)>m$, and that $\short(W_t)\geq \short(W)+t=2\short(W)-3>m$. 

By Proposition \ref{prop:replacing-with-sums-of-powers-of-g-cycles}, there exists a unitriangular automorphism $\varphi:\RA{\Qtau}\rightarrow\RA{\Qtau}$ of depth greater than $m-3$ that makes $(\Qtau,\Staux+U+\sum_{i=1}^tZ_i+W_{t})$ right-equivalent to $(\Qtau,\Staux+U+\xi)$ for some potential $\xi\in\RA{\Qtau}$ that involves only powers of $g$-cycles and satisfies $\short(\xi)\geq\short\left(\sum_{i=1}^tZ_i+W_{t}\right)>m$.

From the two previous paragraphs we deduce that the automorphism $\Pi:=\varphi\circ\zeta$ satisfies the desired conclusion of Corollary \ref{coro:essential}.
\end{proof}

  \section{Once-punctured surfaces}

\label{sec:once-punctured-surfaces}
In \cite{Labardini1}  and \cite{Labardini4}, the second author showed that the potentials $\Staux$ are well behaved with respect to flips and mutations, in the sense that if two triangulations are related by a flip, then the associated QPs are related by the corresponding QP-mutation. In this section, we show that for once-punctured closed surfaces the same result is true for a wider class of potentials. Namely, given a triangulation $\tau$ of a once-punctured close surface of positive genus $(\Sigma,\mathbb{M})$, a scalar $x\neq 0$ and a positive integer $n$, we define a potential $\Stauxn$ as
\begin{align*}
\Stauxn = \Ttau + x\Gsymbol(p)^n,
\end{align*}
where $p$ is the only puncture  in $\surf$.

\begin{thm}\label{thm:flip->QP-mut-arbitrary-powers}
Let $\surf$ be a once-punctured closed surface of positive genus, $n$ be any positive integer, let $x\in K$ be any scalar. If $\tau$ and $\sigma$ are triangulations of $\surf$ that are related by the flip of an arc $k\in\tau$, then the QPs $\mu_k(\Qtau,\Stauxn)$ and $(\Qsigma,\Ssigmaxn)$ are right-equivalent.
\end{thm}

\begin{proof}
Let $a_i,b_i,c_i,~i=1,2$, be the arrows in the two triangles with one side $k$ as in the figure below.

\begin{figure}[H]
\begin{center}
\begin{tikzpicture}[y=.3cm, x=.3cm,font=\normalsize]

\draw (0,0) to (6.92,12);
\draw (0,0) to (-6.92,12);
\draw (-6.92,12) to (6.92,12);
\draw (0,24) to (6.92,12);
\draw (0,24) to (-6.92,12);
\filldraw[fill=black!100,draw=black!100] (0,24) circle (3pt)    node[anchor=north] {};
\filldraw[fill=black!100,draw=black!100] (0,0) circle (3pt)    node[anchor=north] {};
\filldraw[fill=black!100,draw=black!100] (-6.92,12) circle (3pt)    node[anchor=north] {};
\filldraw[fill=black!100,draw=black!100] (6.92,12) circle (3pt)    node[anchor=north] {};

\draw [<-, thick] (-2.7,6) to (2.7,6);
\draw [->, thick] (-2.7,6.3) to (-0.3,11.7);
\draw [<-, thick] (2.7,6.3) to (0.3,11.7);

\draw [->, thick] (-2.7,18) to (2.7,18);
\draw [<-, thick] (-2.7,17.7) to (-0.3,12.3);
\draw [->, thick] (2.7,17.7) to (0.3,12.3);

\node [below] at (0,6) {$a_2$};
\node [left] at (-1.5,9) {$c_2$};
\node [right] at (1.5,9) {$b_2$};

\node [above] at (0,18) {$a_1$};
\node [left] at (-1.5,15) {$b_1$};
\node [right] at (1.5,15) {$c_1$};

\node [above] at (0,12.4) {$k$};
\end{tikzpicture}\\
\caption{The two triangles with one side $k$.}
\label{fig:triangles one side k}
\end{center}
\end{figure}
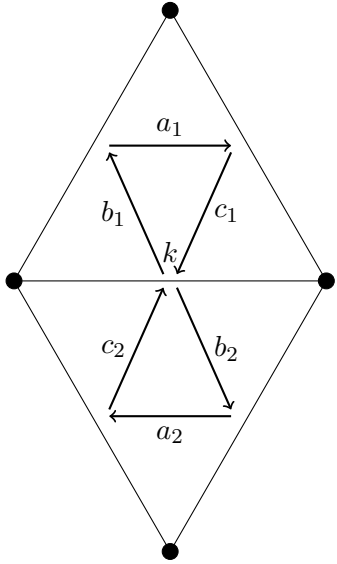

Up to rotation we can write $\Gsymbol(p)=a_1Aa_2B$. Notice that $b_2c_1,b_1c_2$ are factors of $\Gsymbol(p)$, but $b_1c_1$ and $b_2c_2$ are not. The potential $\widetilde{\mu}_k(\Stauxn)$ is cyclically equivalent to
\begin{eqnarray*}
[\Ttau] + x([a_1Aa_2B])^n +c_1^*b_2^*[b_2c_1]+c_2^*b_1^*[b_1c_2]+ c_1^*b_1^*[b_1c_1]+c_2^*b_2^*[b_2c_2] &=& \\
\Tsigma +a_1[b_1c_1]+a_2[b_2c_2]+ c_1^*b_1^*[b_1c_1]+c_2^*b_2^*[b_2c_2]+ x(a_1[A]a_2[B])^n , &&
\end{eqnarray*}
where the paths $[A],[B]$ are the result of replacing $b_2c_1,b_1c_2$ in $A,B$ by $[b_2c_1],[b_1c_2]$, respectively.

\begin{figure}[H]
\begin{center}
\begin{tikzpicture}[y=.3cm, x=.3cm,font=\normalsize]
\filldraw[fill=black!100,draw=black!100] (0,0) circle (3pt)    node[anchor=north] {};
\draw (-7,-7) to (0,0);
\draw (7,7) to (0,0);
\draw (7,-7) to (0,0);
\draw (-7,7) to (0,0);
\draw [->, thick] (-2.7,-3) to (2.7,-3);
\draw [<-, thick] (-2.7,3) to (2.7,3);
\draw [->, thick, dashed] (3,-2.7) to (3,2.7);
\draw [->, thick, dashed] (-3,2.7) to (-3,-2.7);
\node [below] at (0,-3) {$a_1$};
\node [above] at (0,3) {$a_2$};
\node [right] at (3,0) {$B$};
\node [left] at (-3,0) {$A$};
\node at (0,8) {$\Qtau$};

\filldraw[fill=black!100,draw=black!100] (28,0) circle (3pt)    node[anchor=north] {};
\draw (21,-7) to (28,0);
\draw (35,7) to (28,0);
\draw (35,-7) to (28,0);
\draw (21,7) to (28,0);
\draw (28,-7) to (28,0);
\draw (28,7) to (28,0);
\draw [->, thick, dashed] (31,-2.7) to (31,2.7);
\draw [->, thick, dashed] (25,2.7) to (25,-2.7);
\draw [<-, thick] (28.3,3) to (30.7,3);
\draw [<-, thick] (25.3,3) to (27.7,3);
\draw [->, thick] (28.3,-3) to (30.7,-3);
\draw [->, thick] (25.3,-3) to (27.7,-3);
\node [below] at (26.3,-3) {$b_1^*$};
\node [below] at (29.7,-3) {$c_1^*$};
\node [above] at (26.3,3) {$c_2^*$};
\node [above] at (29.7,3) {$b_2^*$};
\node [right] at (31,0) {$[B]$};
\node [left] at (25,0) {$[A]$};
\node at (28,8) {$\Qsigma$};
\end{tikzpicture}\\
\caption{The cycle on $\Qtau$ and $\Qsigma$ surrounding the puncture.}
\label{fig:Arrows surrounding the puncture}
\end{center}
\end{figure}
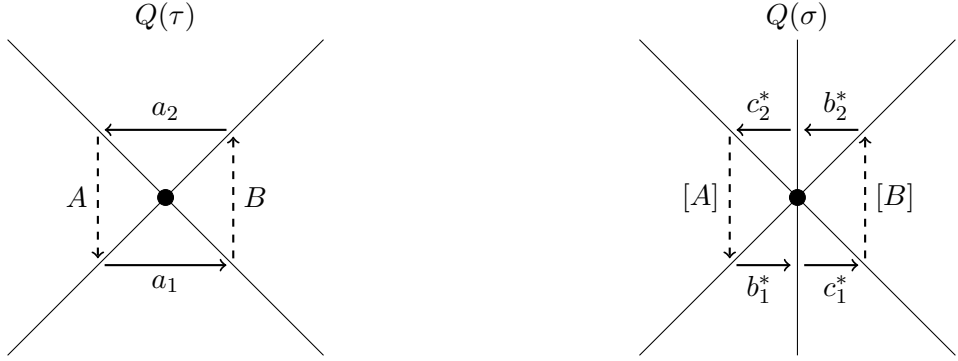

We define $R$-algebra homomorphisms $\varphi_1,\varphi_2:\RA{\Qtau}\rightarrow\RA{\Qtau}$ by means of the rules
\begin{align*}
\varphi_1(a_1) =& a_1-c_1^*b_1^*;\\
\varphi_2([b_1c_1]) =& [b_1c_1] - x\sum_{j=0}^{n-1}(-1)^j[A]a_2[B]((a_1-c_1^*b_1^*)[A]a_2[B])^{n-j-1}(c_1^*b_1^*[A]a_2[B])^j.
\end{align*}
Applying $\varphi_1$ to $\widetilde{\mu}_k(\Stauxn)$ we get
\begin{align*}
\varphi_1(\widetilde{\mu}_k(\Stauxn)) \sim_{\operatorname{cyc}}&
\Tsigma +a_1[b_1c_1]+a_2[b_2c_2] +c_2^*b_2^*[b_2c_2]+ x((a_1-c_1^*b_1^*)[A]a_2[B])^n\\
\sim_{\operatorname{cyc}}&
\Tsigma +a_1[b_1c_1]+a_2[b_2c_2]+c_2^*b_2^*[b_2c_2]+ x(-1)^n(c_1^*b_1^*[A]a_2[B])^n\\
&+x\sum_{j=0}^{n-1}(-1)^ja_1[A]a_2[B]((a_1-c_1^*b_1^*)[A]a_2[B])^{n-j-1}(c_1^*b_1^*[A]a_2[B])^j.
\end{align*}
The potential $\varphi_2\varphi_1(\widetilde{\mu}_k(\Stauxn))$ is cyclically equivalent to
\begin{align*}
\varphi_2\varphi_1(\widetilde{\mu}_k(\Stauxn)) \sim_{\operatorname{cyc}}&
\Tsigma +a_1[b_1c_1]+a_2[b_2c_2]+c_2^*b_2^*[b_2c_2]+ x(-1)^n(c_1^*b_1^*[A]a_2[B])^n.
\end{align*}
In an analogous way, we define $R$-algebra homomorphisms $\varphi_3,\varphi_4:\RA{\Qtau}\rightarrow\RA{\Qtau}$ by means of the rules
\begin{align*}
\varphi_3(a_2) =& a_2-c_2^*b_2^*;\\
\varphi_4([b_2c_2]) =& [b_2c_2] - x(-1)^n\sum_{j=0}^{n-1}(-1)^j[B]c_1^*b_1^*[A]((a_2-c_2^*b_2^*)[B]c_1^*b_1^*[A])^{n-j-1}(c_2^*b_2^*[B]c_1^*b_1^*[A])^j.
\end{align*}
We obtain
\begin{align*}
\varphi_4\varphi_3\varphi_2\varphi_1(\widetilde{\mu}_k(\Stauxn)) \sim_{\operatorname{cyc}}&
\Tsigma +a_1[b_1c_1]+a_2[b_2c_2]+ x(c_1^*b_1^*[A]c_2^*b_2^*[B])^n\\
\sim_{\operatorname{cyc}}& \Ssigmaxn +a_1[b_1c_1]+a_2[b_2c_2].
\end{align*}
Therefore, the QPs $\mu_k(\Qtau,\Stauxn)$ and $(\Qsigma,\Ssigmaxn)$ are right-equivalent.
\end{proof}

\begin{remark} \begin{enumerate}
\item For once-punctured closed surfaces, Theorem \ref{thm:flip->QP-mut-arbitrary-powers} constitutes a generalization of the second author's \cite[Theorem 30]{Labardini1} and \cite[Theorem 8.1]{Labardini4}.
\item It was observed by Ladkani \cite[Proposition 3.1]{Ladkani2} that the proof of \cite[Theorem 30]{Labardini1} can be applied without change to produce a proof of Theorem \ref{thm:flip->QP-mut-arbitrary-powers} above for $x=0$.
\item In his Master thesis \cite{Geuenich-master-thesis}, the first author of this paper proved Theorem \ref{thm:flip->QP-mut-arbitrary-powers} for the once-punctured torus and $x\neq 0$.
\item Motivated by the first author's Master thesis, the third author proved Theorem \ref{thm:flip->QP-mut-arbitrary-powers} in his Undergraduate thesis \cite{Miranda-undergrad-thesis}.
\end{enumerate}
\end{remark}

\begin{prop}\label{prop:unbounded-dimensions} Let $\surf$ be a once-punctured closed surface of positive genus, $\tau$ a triangulation of $\surf$, and $x\in K$ a non-zero scalar. If the characteristic of the field $K$ is zero, then
$$
\dim_K(\mathcal{P}(\Qtau,\Stauxn))<\infty \ \ \ \text{and} \ \ \ 
\lim_{n\to\infty}\dim_K(\mathcal{P}(\Qtau,\Stauxn))=\infty.
$$
\end{prop}

\begin{proof}
For the proof of finite-dimensionality we follow ideas suggested by Ladkani in his proof of \cite[Proposition 4.2]{Ladkani}, whereas our proof that the limits of the dimensions is $\infty$ follows ideas that appear in the first author's Master thesis.

First, note that when we compute the cyclic derivative of $\Stauxn$ with respect to an arrow $\alpha$, we get
\begin{equation}\label{eq:ciclyc derivative}
\partial_\alpha(\Stauxn)=f^2(\alpha)f(\alpha)+xnG(nm_\alpha-1,g(\alpha)).
\end{equation}
So, $f^2(\alpha)f(\alpha)$ and $-xnG(nm_\alpha-1,g(\alpha))$ become equal in the Jacobian algebra $\mathcal{P}(\Qtau,\Stauxn)$. 

Every $fg$-path of length three has the form $f^2(\alpha)f(\alpha)g^{-1}f(\alpha)$ or $gf^2(\alpha)f^2(\alpha)f(\alpha)$ for some arrow $\alpha$, and it is hence equal to $-xnG(nm_\alpha-1,g(\alpha))g^{-1}f(\alpha)=-xnG(nm_\alpha-3,g^3(\alpha))g^2(\alpha)g(\alpha)f^{-1}g(\alpha)$ or $-xngf^2(\alpha)G(nm_\alpha-1,g(\alpha))=-xnfg^{-1}(\alpha)g^{-1}(\alpha)g^{-2}(\alpha)G(nm_\alpha-3,g(\alpha))$ in $\mathcal{P}(\Qtau,\Stauxn)$. Thus every $fg$-path of length three is equal in $\mathcal{P}(\Qtau,\Stauxn)$ to another $fg$-path of length greater than three. In the same vein, an easy inductive argument shows that, in the Jacobian algebra, every $fg$-path is to an arbitrarily long $fg$-path, and therefore equal to $0\in \mathcal{P}(\Qtau,\Stauxn)$.

Any $f$-path $F(r,f(\beta))=F(r-2,\beta)f^2(\beta)f(\beta)$ of length $r$ greater than three, is equal to the $fg$-path $-xnF(r-2,\beta)G(nm_\beta-1,g(\beta))$ in $\mathcal{P}(\Qtau,\Stauxn)$, and in this way, to $0$. Furthermore, any $g$-path of the form $G(r,g(\beta))=G(r-nm_\beta+1,\beta)G(nm_\beta-1,g(\beta))$ with length greater than $nm_\beta$, is equal to the $fg$-path $x^{-1}n^{-1}G(r-nm_\beta+1,\beta)f^2(\beta)f(\beta)$, hence equal to $0$ in the Jacobian algebra. Notice that here, we have used that $\field$ is a field of characteristic zero.

Thus far, we have shown that every path of length greater than $nm_\alpha$ is equivalent to 0 in the Jacobian algebra $\mathcal{P}(\Qtau,\Stauxn)$, and therefore the latter has finite dimension.

On the other hand, as the cyclic derivative of $\Stauxn$ with respect to any arrow $\alpha$ is equal to the sum of an $f$-path of length two and a scalar multiple of a $g$-path of length $nm_\alpha-1$ \eqref{eq:ciclyc derivative}, and since no $g$-path is a multiple of any $f$-path of length greater than one, we conclude that for any $a,b \in \RA{\Qtau}$, no $g$-path of length smaller than  $nm_\alpha-1$ appears in the expression of the element $a\partial_\alpha(\Stauxn) b$ as a possibly infinite sum of paths on the quiver $\Qtau$. From this, it follows that no finite linear combination of $g$-paths of lengths smaller than $nm_\alpha-1$ can be written as a limit of finite sums of elements of the form $a\partial_\alpha(\Stauxn) b$, i.e., the set of $g$-paths of length smaller than $nm_\alpha-1$ is linearly independent in the Jacobian algebra $\mathcal{P}(\Qtau,\Stauxn)$. Therefore, $\dim_K(\mathcal{P}(\Qtau,\Stauxn))\geq nm_\alpha-2$.
\end{proof}

\begin{coro} Over a field of characteristic zero, the quiver of any triangulation of a once-punctured closed surface of positive genus admits infinitely many non-degenerate potentials up to weak right-equivalence.
\end{coro}

\begin{remark} 
\begin{enumerate}
\item In the case of the once-punctured torus, Proposition \ref{prop:unbounded-dimensions}, was proved by the first author in his Master thesis \cite{Geuenich-master-thesis}.
\item In his Undergraduate thesis \cite{Miranda-undergrad-thesis}, the third author has computed an actual $K$-vector space basis of $\mathcal{P}(\Qtau,\Stauxn)$ for each $n\geq 1$, showing in particular that different values of $n$ never yield Jacobian algebras with the same dimension. This implies that different values of $n$ always yield potentials that are not weakly right-equivalent.
\end{enumerate}
\end{remark}
  \section{Twice-punctured surfaces}

\label{sec:twice-punctured-surfaces}

In this section we prove part \eqref{item:main-thm-two-punctures} of Theorem \ref{main-thm-intro}, namely:

\begin{thm}\label{thm:uniqueness-two-punctures} Let $\surf$ be a twice-punctured closed surface of positive genus, and let $\tau$ be any (tagged) triangulation of $\surf$. Over an algebraically closed field, any two non-degenerate potentials on the quiver $\Qtau$ are weakly right-equivalent.
\end{thm}

Since any two ideal triangulations of $\surf$ are related by a finite sequence of flips (see \cite{Mosher}), the first paragraphs of the proof of \cite[Lemma 8.5]{GLFS} imply that the mere exhibition of a single triangulation $\tau$ of $\surf$, with $\Qtau$ having only one weak right equivalence class of non-degenerate potentials, suffices in order to prove Theorem \ref{thm:uniqueness-two-punctures}.

\begin{ex}\label{ex:8g-4g-triangulation} Figure \ref{Fig:8g-4g-triangulation} sketches a triangulation $\tau$ of a positive-genus twice-punctured surface with empty boundary. The triangulation is easily seen to satisfy \eqref{eq:valency-at-least-4} and \eqref{eq:no-double-arrows}. Note that the puncture $p$ has valency $8g$ and the other puncture $q$ has valency $4g$.
\begin{figure}[H]
\begin{center}
\begin{tikzpicture}[y=.6cm, x=.6cm,font=\normalsize]
\node [above] at (0,6.01) {$1$};
\draw (-2.49,6.01) -- (2.49,6.01);

\node [right] at (4.25,4.25) {$2$};
\draw (2.49,6.01) -- (6.01,2.49);

\draw  (6.01,2.49) -- (6.01,-2.49);
\node [right] at (6.01,0) {$1$};

\node [right] at (4.25,-4.25) {$2$};
\draw (2.49,-6.01) -- (6.01,-2.49);

\node [left] at (-6.01,0) {$2g-1$};
\draw (-6.01,-2.49) -- (-6.01,2.49);

\node [left] at (-4.25,4.25) {$2g$};
\draw (-2.49,6.01) -- (-6.01,2.49);

\node [left] at (-4.25,-4.25) {$2g$};
\draw (-2.49,-6.01) -- (-6.01,-2.49);

\draw (-2.49,-6.01) -- (0,0);
\draw (2.49,-6.01) -- (0,0);
\draw (-2.49,6.01) -- (0,0);
\draw (2.49,6.01) -- (0,0);

\draw (6.01,2.49) -- (0,0);
\draw (-6.01,2.49) -- (0,0);
\draw (-6.01,-2.49) -- (0,0);
\draw (6.01,-2.49) -- (0,0);

\node [above] at (0,0.3) {$q$};

\node [right] at (6.01,2.49) {$p$};
\node [left] at (-6.01,2.49) {$p$};
\node [right] at (6.01,-2.49) {$p$};
\node [left] at (-6.01,-2.49) {$p$};
\node [above] at (2.49,6.01) {$p$};
\node [above] at (-2.49,6.01) {$p$};
\node [below] at (2.49,-6.01) {$p$};
\node [below] at (-2.49,-6.01) {$p$};

\filldraw[fill=black!100,draw=black!100] (1,-6.01) circle (1pt)    node[anchor=north] {};
\filldraw[fill=black!100,draw=black!100] (-1,-6.01) circle (1pt)    node[anchor=north] {};
\filldraw[fill=black!100,draw=black!100] (0,-6.01) circle (1pt)    node[anchor=north] {};
\filldraw[fill=black!100,draw=black!100] (0,0) circle (3pt)    node[anchor=north] {};
\filldraw[fill=black!100,draw=black!100] (6.01,2.49) circle (3pt)    node[anchor=north] {};
\filldraw[fill=black!100,draw=black!100] (6.01,-2.49) circle (3pt)    node[anchor=north] {};
\filldraw[fill=black!100,draw=black!100] (-6.01,-2.49) circle (3pt)    node[anchor=north] {};
\filldraw[fill=black!100,draw=black!100] (-6.01,2.49) circle (3pt)    node[anchor=north] {};
\filldraw[fill=black!100,draw=black!100] (2.49,6.01) circle (3pt)    node[anchor=north] {};
\filldraw[fill=black!100,draw=black!100] (2.49,-6.01) circle (3pt)    node[anchor=north] {};
\filldraw[fill=black!100,draw=black!100] (-2.49,-6.01) circle (3pt)    node[anchor=north] {};
\filldraw[fill=black!100,draw=black!100] (-2.49,6.01) circle (3pt)    node[anchor=north] {};

\node [above] at (1.40,3) {$2g+1$};
\node [above] at (3,1.24) {$2g+2$};
\node [below] at (1.24,-3) {$2g+4$};
\node [below] at (3,-1.24) {$2g+3$};

\node [above] at (-1.24,3) {$6g$};
\node [above] at (-3,1.24) {$6g-1$};
\node [below] at (-1.24,-3) {$6g-3$};
\node [below] at (-3,-1.24) {$6g-2$};
\end{tikzpicture}\\
\caption{A triangulation $\tau$ of a twice-punctured closed surface $(\Sigma,\mathbb{M})$ of positive-genus.}
\label{Fig:8g-4g-triangulation}
\end{center}
\end{figure}
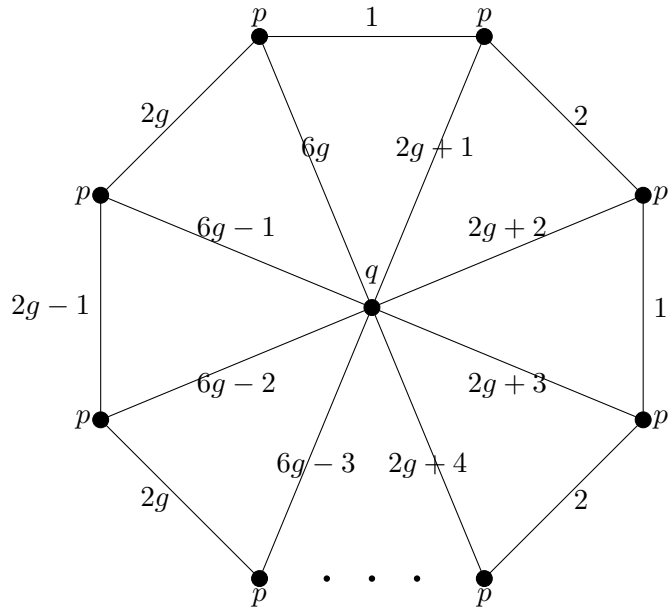

\begin{figure}[H]
\begin{center}
\begin{tikzpicture}[y=.6cm, x=.6cm,font=\normalsize]
\node [above] at (0,6.01) {$1$};
\draw (-2.49,6.01) -- (2.49,6.01);

\node [right] at (4.25,4.25) {$2$};
\draw (2.49,6.01) -- (6.01,2.49);

\draw  (6.01,2.49) -- (6.01,-2.49);
\node [right] at (6.01,0) {$1$};

\node [right] at (4.25,-4.25) {$2$};
\draw (2.49,-6.01) -- (6.01,-2.49);

\node [left] at (-6.01,0) {$2g-1$};
\draw (-6.01,-2.49) -- (-6.01,2.49);

\node [left] at (-4.25,4.25) {$2g$};
\draw (-2.49,6.01) -- (-6.01,2.49);

\node [left] at (-4.25,-4.25) {$2g$};
\draw (-2.49,-6.01) -- (-6.01,-2.49);

\draw (-2.49,-6.01) -- (0,0);
\draw (2.49,-6.01) -- (0,0);
\draw (-2.49,6.01) -- (0,0);
\draw (2.49,6.01) -- (0,0);

\draw (6.01,2.49) -- (0,0);
\draw (-6.01,2.49) -- (0,0);
\draw (-6.01,-2.49) -- (0,0);
\draw (6.01,-2.49) -- (0,0);

\node [below] at (0,-0.3) {$q$};

\node [right] at (6.01,2.49) {$p$};
\node [left] at (-6.01,2.49) {$p$};
\node [right] at (6.01,-2.49) {$p$};
\node [left] at (-6.01,-2.49) {$p$};
\node [above] at (2.49,6.01) {$p$};
\node [above] at (-2.49,6.01) {$p$};
\node [below] at (2.49,-6.01) {$p$};
\node [below] at (-2.49,-6.01) {$p$};

\filldraw[fill=black!100,draw=black!100] (1,-6.01) circle (1pt)    node[anchor=north] {};
\filldraw[fill=black!100,draw=black!100] (-1,-6.01) circle (1pt)    node[anchor=north] {};
\filldraw[fill=black!100,draw=black!100] (0,-6.01) circle (1pt)    node[anchor=north] {};
\filldraw[fill=black!100,draw=black!100] (0,0) circle (3pt)    node[anchor=north] {};
\filldraw[fill=black!100,draw=black!100] (6.01,2.49) circle (3pt)    node[anchor=north] {};
\filldraw[fill=black!100,draw=black!100] (6.01,-2.49) circle (3pt)    node[anchor=north] {};
\filldraw[fill=black!100,draw=black!100] (-6.01,-2.49) circle (3pt)    node[anchor=north] {};
\filldraw[fill=black!100,draw=black!100] (-6.01,2.49) circle (3pt)    node[anchor=north] {};
\filldraw[fill=black!100,draw=black!100] (2.49,6.01) circle (3pt)    node[anchor=north] {};
\filldraw[fill=black!100,draw=black!100] (2.49,-6.01) circle (3pt)    node[anchor=north] {};
\filldraw[fill=black!100,draw=black!100] (-2.49,-6.01) circle (3pt)    node[anchor=north] {};
\filldraw[fill=black!100,draw=black!100] (-2.49,6.01) circle (3pt)    node[anchor=north] {};

\draw [<-, thick] (-1.21,3) -- (1.21,3);
\draw [<-, thick] (1.21,3.3) -- (0.15,5.7);
\draw [->, thick] (-1.21,3.3) -- (-0.15,5.7);

\draw [->, thick] (3,-1.21) -- (3,1.21);
\draw [->, thick] (3.3,1.21) -- (5.7,0.15);
\draw [<-, thick] (3.3,-1.21) -- (5.7,-0.15);

\draw [<-, thick] (-3,-1.21) -- (-3,1.21);
\draw [<-, thick] (-3.3,1.21) -- (-5.7,0.15);
\draw [->, thick] (-3.3,-1.21) -- (-5.7,-0.15);

\draw [<-, thick] (1.23,3) -- (3,1.23);
\draw [->, thick] (1.28,3.2) -- (4.17,4.27);
\draw [<-, thick] (3.2,1.28) -- (4.27,4.17);

\draw [->, thick] (-1.23,3) -- (-3,1.23);
\draw [<-, thick] (-1.28,3.2) -- (-4.17,4.27);
\draw [->, thick] (-3.2,1.28) -- (-4.27,4.17);

\draw [->, thick] (1.23,-3) -- (3,-1.23);
\draw [<-, thick] (1.28,-3.2) -- (4.17,-4.27);
\draw [->, thick] (3.2,-1.28) -- (4.27,-4.17);

\draw [<-, thick] (-1.23,-3) -- (-3,-1.23);
\draw [->, thick] (-1.28,-3.2) -- (-4.17,-4.27);
\draw [<-, thick] (-3.2,-1.28) -- (-4.27,-4.17);

\node [below] at (0,3) {$a_1$};
\node [right] at (0.68,4.5) {$b_1$};
\node [left] at (-0.68,4.5) {$c_1$};

\node [left] at (3,0) {$a_3$};
\node [above] at (4.5,0.68) {$c_3$};
\node [below] at (4.5,-0.68) {$b_3$};

\node [right] at (-3,0) {$a_{4g-1}$};
\node [above] at (-4.5,0.68) {$b_{4g-1}$};
\node [below] at (-4.5,-0.68) {$c_{4g-1}$};

\node [left] at (2.11,2.11) {$a_2$};
\node [above] at (2.72,3.73) {$c_2$};
\node [right] at (3.73,2.72) {$b_2$};

\node [left] at (2.11,-2.11) {$a_4$};
\node [below] at (2.72,-3.73) {$b_4$};
\node [right] at (3.73,-2.72) {$c_4$};

\node [right] at (-2.11,2.11) {$a_{4g}$};
\node [above] at (-2.72,3.73) {$b_{4g}$};
\node [left] at (-3.73,2.72) {$c_{4g}$};

\node [right] at (-2.11,-2.11) {$a_{4g-2}$};
\node [below] at (-2.72,-3.73) {$c_{4g-2}$};
\node [left] at (-3.73,-2.72) {$b_{4g-2}$};
\end{tikzpicture}\\
\caption{The associated quiver $Q(\tau)$ to the triangulation $\tau$.}
\label{Fig:Q-tau-2-punctured}
\end{center}
\end{figure}
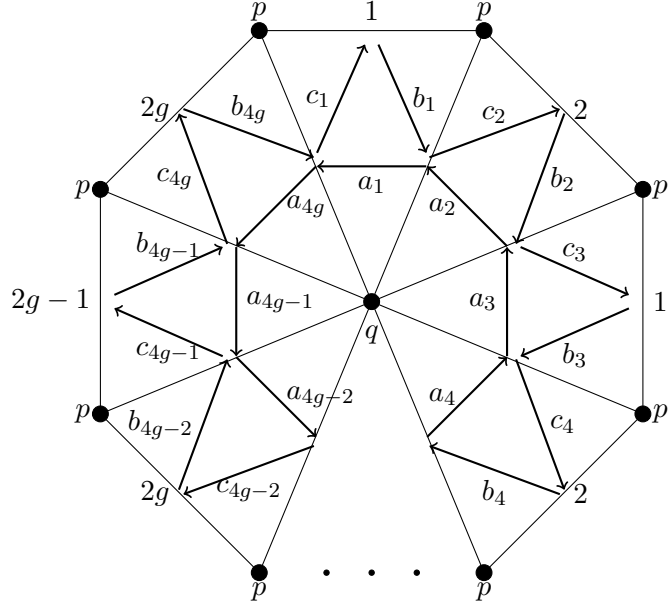

\end{ex}

\begin{lemma}\label{lemma:last-cancellation--g-powers} Let $\surf$ be a twice-punctured closed surface of positive genus, and let $\tau$ be the triangulation of $\surf$ depicted in Figure \ref{Fig:8g-4g-triangulation}. If $V\in\RA{\Qtau}$ is a potential involving only $\geq 2$-powers of $g$-cycles, then $(\Qtau,\Staux+V)$ is right-equivalent to $(\Qtau,\Staux)$ for any choice $\xx=(x_p,x_q)$ of non-zero scalars.
\end{lemma}

\begin{proof}
Let $g$ be the genus of $\surf$. Then
\begin{equation*}
V\sim_{\operatorname{cyc}} \sum_{n=2}^\infty\nu_{p,n}(\Gsymbol(p))^n+\sum_{n=2}^\infty\nu_{q,n}(\Gsymbol(q))^n
\end{equation*}
for some scalars $\nu_{p,n}$ and $\nu_{q,n}$ for $n\geq 2$.
Note that $\short(V)\geq 2\val_\tau(q)=8g$. 

\begin{claim}\label{claim:sequence-Vm} There exist a sequence $(V_m)_{m=8g}^{\infty}$ of potentials on $\Qtau$, and a sequence $(\varphi_m)_{m=8g}^{\infty}$ of unitriangular $R$-algebra automorphisms of $\RA{\Qtau}$, satisfying the following properties:
\begin{enumerate}
\item $V_{8g}=V$;
\item $\lim_{m\to\infty}\depth(\varphi_m)=\infty$;
\item for every $m\geq8g$:
\begin{enumerate}
\item $\varphi_{m}$ is a right-equivalence $(\Qtau,\Staux+V_m)\rightarrow(\Qtau,\Staux+V_{m+1})$;
\item $V_m$ involves only $\geq2$-powers of $g$-cycles;
\item $\short(V_m)\geq m$.
\end{enumerate}
\end{enumerate}
\end{claim}

\begin{proof}[Proof of Claim \ref{claim:sequence-Vm}] Start by setting $V_{8g}=V$. Let $a_p$ (resp. $a_q$) be an arrow lying in the $g$-orbit that surrounds $p$ (resp. $q$). Suppose that for a fixed value of $m\geq 8g$ we have already defined a potential $V_m$ involving only $\geq 2$-powers of $g$-cycles and satisfying $\short(V_m)\geq m$. We shall use $V_m$ to define $V_{m+1}$ and $\varphi_{m}$. Write:
\begin{equation*}
V_m\sim_{\operatorname{cyc}} \sum_{n=2}^\infty\lambda_{p,n}(\Gsymbol(a_p))^n+\sum_{n=2}^\infty\lambda_{q,n}(\Gsymbol(a_q))^n
\end{equation*}
with $\lambda_{p,n},\lambda_{q,n}\in K$ for $n\geq 2$. Set $r_{p,m}$ (resp. $r_{q,m}$) 
to be the first value of $n$ for which $\lambda_{p,n}\neq 0$ (resp. $\lambda_{q,n}\neq 0$) if such an $n$ exists, and $\infty$ if such an $n$ does not exist. Note that $\short(V_m)=\min(8gr_{p,m},4gr_{q,n})\geq 8g$.

Define an $R$-algebra homomorphism $\Upsilon_{p,m}:\RA{\Qtau}\rightarrow\RA{\Qtau}$ by means of the rule
\begin{eqnarray*}
\Upsilon_{p,m}&:& a_p\mapsto a_p-\frac{\lambda_{p,r_{p,m}}}{x_p}a_p(\Gsymbol(a_p))^{r_{p,m}-1}.
\end{eqnarray*}
Since $r_{p,m}-1>0$, $\Upsilon_{p,n}$ is a unitriangular automorphism, its depth is $8g(r_{p,m}-1)$. Direct computation shows that
\begin{eqnarray*}
\Upsilon_{p,m}(\Staux+V_m) &\sim_{\operatorname{cyc}}& 
\Staux+U+W,
\end{eqnarray*}
where
\begin{eqnarray*}
U&=& -\lambda_{p,r_{p,m}}(\Gsymbol(a_p))^{r_{p,m}}+\Upsilon_{p,m}\left(\sum_{n=r_{p,m}}^\infty\lambda_{p,n}(\Gsymbol(a_p))^n\right)+\sum_{n=r_{q,m}}^\infty\lambda_{q,n}(\Gsymbol(a_q))^n,\\
W&=&-\frac{\lambda_{p,r_{p,m}}}{x_p}f(a_p)a_p(\Gsymbol(a_p))^{r_{p,m}-1}f^2(a_p).
\end{eqnarray*}
Note that $\short(U)\geq m$ and $2\short(W)-3=2*8g(r_{p,m}-1)+3\geq 8gr_{p,m}+3>8gr_{p,m}\geq m$. So, applying Corollary \ref{coro:essential}, we see that there exists a unitriangular $R$-algebra automorphism $\Pi_{p,m}$ of $\RA{\Qtau}$ that has depth at least $\min(m-3,8g(r_{p,m}-1))$ and serves as a right-equivalence between $\Staux+U+W$ and $\Staux+U+\xi$ for some potential $\xi$ that involves only positive powers of $g$-cycles and satisfies $\short(\xi)>m\geq 8g$. These last inequalities imply that, actually, $\xi$ involves only $\geq2$-powers of $g$-cycles.

Now, we can definitely write
\begin{equation}
U\sim_{\operatorname{cyc}} \sum_{n=r_{p,m}+1}^\infty\kappa_{p,n}(\Gsymbol(a_p))^n+\sum_{n=r_{q,m}}^\infty\lambda_{q,n}(\Gsymbol(a_q))^n
\end{equation}
for some scalars $\kappa_{p,n}\in K$. Define an $R$-algebra homomorphism $\Upsilon_{q,m}:\RA{\Qtau}\rightarrow\RA{\Qtau}$ by means of the rule
\begin{eqnarray*}
\Upsilon_{q,m} &:& a_q\mapsto a_q-\frac{\lambda_{q,n}}{x_q}a_q(\Gsymbol(a_q))^{r_{q,m}-1}.
\end{eqnarray*}
Since $r_{q,m}-1>0$, $\Upsilon_{q,m}$ is a unitriangular automorphism, its depth is $4g(r_{q,m}-1)$. Direct computation shows that
\begin{eqnarray*}
\Upsilon_{q,m}(\Staux+U+\xi)&\sim_{\operatorname{cyc}}&
\Staux+U'+W',
\end{eqnarray*}
where
\begin{eqnarray*}
U' &=& -\lambda_{q,r_{q,m}}(\Gsymbol(a_q))^{r_{q,m}}+\sum_{n=r_{p,m}+1}^\infty\kappa_{p,n}(\Gsymbol(a_p))^n+\Upsilon_{q,n}\left(\sum_{n=r_{q,m}}^\infty\lambda_{q,n}(\Gsymbol(a_q))^n\right)+\Upsilon_{q,m}(\xi),\\
W'&=&-\frac{\lambda_{q,r_{q,m}}}{x_q}f(a_q)a_q(\Gsymbol(a_q))^{r_{q,m}-1}f^2(a_q).
\end{eqnarray*}
Note that $\short(U')>m$ and $2\short(W')-3=2*4g(r_{q,m}-1)+3\geq 4gr_{q,m}+3>4gr_{q,m}\geq m$. So, applying Corollary \ref{coro:essential}, we see that there exists a unitriangular $R$-algebra automorphism $\Pi_{q,m}$ of $\RA{\Qtau}$ that has depth at least $\min(m-3,4g(r_{q,m}-1))$ and serves as a right-equivalence between $\Staux+U'+W'$ and $\Staux+U'+\xi'$ for some potential $\xi'$ that involves only positive powers of $g$-cycles and satisfies $\short(\xi')>m\geq 8g$. These last inequalities imply that, actually, $\xi'$ involves only $\geq2$-powers of $g$-cycles.

It is clear that $U'$ involves only positive powers of $g$-cycles; this powers are actually greater than 1 because $\short(U')>m\geq 8g$. So, if we set $V_{m+1}=U'+\xi'$ and $\varphi_m=\Pi_{q,m}\Upsilon_{q,m}\Pi_{p,m}\Upsilon_{p,m}$, we see that $\varphi_m$ is a right-equivalence $(\Qtau,\Staux+V_m)\rightarrow(\Qtau,\Staux+V_{m+1})$, that $V_{m+1}$ involves only $\geq2$-powers of $g$-cycles, and that $\short(V_{m+1})\geq m+1$.

From the previous paragraph we deduce that the sequences $(V_m)_{m\geq 8g}$ and $(\varphi_m)_{m\geq 8g}$ satisfy the third condition stated in Claim \ref{claim:sequence-Vm}. Moreover, since $m\leq \short(V_m)=\min(8gr_{p,m},4gr_{q,m})$ for every $m\geq 8g$, we deduce that $\lim_{m\to\infty}r_{p,m}=\infty=\lim_{m\to\infty}r_{q,m}$. This and the inequalities
\begin{eqnarray*}
\depth(\varphi_m) &\geq &
\min(\depth(\Pi_{q,m}),\depth(\Upsilon_{q,m}),\depth(\Pi_{p,m}),\depth(\Upsilon_{p,m}))\\
&\geq&
\min(\min(m-3,4g(r_{q,m}-1)), 4g(r_{q,m}-1), \min(m-3,8g(r_{p,m}-1)),8g(r_{p,m}-1))
\end{eqnarray*}
imply that $\lim_{m\to\infty}\depth(\varphi_m)=\infty$.

Our Claim \ref{claim:sequence-Vm} is proved.
\end{proof}

Lemma \ref{lemma:last-cancellation--g-powers} follows from an obvious combination of Claim \ref{claim:sequence-Vm} and \cite[Lemma 2.4]{Labardini4}.
\end{proof}

\begin{prop}\label{prop:g-cycles-always-appear-2-punctures} Let $\surf$ be a twice-punctured closed surface of positive genus, and let $\tau$ be the triangulation of $\surf$ depicted in Figure \ref{Fig:8g-4g-triangulation}. If $W\in\RA{\Qtau}$ is a potential that involves only positive powers of $g$-cycles and such that $(\Qtau,\Ttau+W)$ is a non-degenerate QP, then $W$ involves each of the $g$-cycles that arise from the two punctures $p$ and $q$ of $\surf$, that is, $\Ttau+W=\Staux+V$ for some choice $\xx=(x_p,x_q)$ of non-zero scalars and some potential $V$ involving only $\geq 2$-powers of $g$-cycles.
\end{prop}

\begin{proof} With the notation of Figures \ref{Fig:8g-4g-triangulation} and \ref{Fig:Q-tau-2-punctured}, let us write 
\begin{eqnarray*}
W&=&ya_1a_2\ldots a_{4g}+A+z\left(\prod_{j=0}^{g-1}b_{4(g-j)}c_{4(g-j)-2}b_{4(g-j)-3}c_{4(g-j)-1}b_{4(g-j)-2}c_{4(g-j)}b_{4(g-j)-1}c_{4(g-j)-3}\right)\\
& &+B,  \text{with}\\
A&=&
\sum_{n=2}^\infty y_n(a_1a_2\ldots a_{4g})^{n} \ \text{and}\\
B&=&
\sum_{n=2}^\infty z_n\left(\prod_{j=0}^{g-1}b_{4(g-j)}c_{4(g-j)-2}b_{4(g-j)-3}c_{4(g-j)-1}b_{4(g-j)-2}c_{4(g-j)}b_{4(g-j)-1}c_{4(g-j)-3}\right)^n.
\end{eqnarray*}

If we set $I=\{2g+1,2g+2,\ldots,6g-1,6g\}$, then $(\Qtau,\Ttau+W)$ and $I$ satisfy the hypotheses of \cite[Proposition 2.4]{GLFS}, and we deduce that $y\neq 0$.

Note that for every $k\in\{1,\ldots,2g-1\}$, the quiver $\widetilde{\mu}_k\widetilde{\mu}_{k-1}\ldots\widetilde{\mu}_2\widetilde{\mu}_1(\Qtau)$ does not have 2-cycles incident to the vertex labelled $k+1$. Therefore, the QP $\mu_{2g}\mu_{2g-1}\ldots\mu_2\mu_1(\Qtau,\Ttau+W)$ is right-equivalent to the reduced part of the QP $\widetilde{\mu}_{2g}\widetilde{\mu}_{2g-1}\ldots\widetilde{\mu}_2\widetilde{\mu}_1(\Qtau,\Ttau+W)$, whose underlying quiver and potential are $\widetilde{\mu}_{2g}\ldots\widetilde{\mu}_1(\Qtau)$ and
\begin{eqnarray*}
\widetilde{\mu}_{2g}\ldots\widetilde{\mu}_1(\Ttau+W)
&=&
\left(\sum_{j=1}^{4g}a_j[b_jc_j]\right)+ya_1\ldots a_{4g}+A+[B]\\
&+&
z\left(\prod_{j=0}^{g-1}[b_{4(g-j)}c_{4(g-j)-2}][b_{4(g-j)-3}c_{4(g-j)-1}][b_{4(g-j)-2}c_{4(g-j)}][b_{4(g-j)-1}c_{4(g-j)-3}]\right)\\
&+&
\left(\sum_{j=1}^{2g}c_j^*b_j^*[b_jc_j]+c_{j+2}^*b_j^*[b_jc_{j+2}]+c_{j}^*b_{j+2}^*[b_{j+2}c_j]+c_{j+2}^*b_{j+2}^*[b_{j+2}c_{j+2}]\right).
\end{eqnarray*}

Consider the QP $(\widetilde{\mu}_{2g}\ldots\widetilde{\mu}_1(\Qtau),\overline{S})$, where
\begin{eqnarray*}
\overline{S} &=&
\left(\sum_{j=1}^{4g}a_j[b_jc_j]\right)+ya_1\ldots a_{4g}+A+
\left(\sum_{j=1}^{2g}c_j^*b_j^*[b_jc_j]+c_{j+2}^*b_{j+2}^*[b_{j+2}c_{j+2}]\right),
\end{eqnarray*}
and let $(Q,S)$ be its reduced part, computed according to the limit process with which Derksen-Weyman-Zelevinsky \cite[Theorem 4.6]{DWZ1} prove their Splitting Theorem. Note the presence of the sum $\sum_{j=1}^{4g}a_j[b_jc_j]$ in $\overline{S}$. Then $Q=Q(\sigma)$, where $\sigma$ is a triangulation that can be obtained from $\tau$ by applying an orientation-preserving homeomorphism of $(\Sigma,\mathbb{M})$ that exchanges $p$ and $q$ (thus $\tau$ and $\sigma$ have the same shape, sketched in Figure \ref{Fig:Q-tau-2-punctured}; see also Example \ref{ex torus} below). Moreover, since no arrow of the form $a_j$ or $[b_jc_j]$ appears in any of the terms of the potential 
\begin{eqnarray*}
W'&:=&
z\left(\prod_{j=0}^{g-1}[b_{4(g-j)}c_{4(g-j)-2}][b_{4(g-j)-3}c_{4(g-j)-1}][b_{4(g-j)-2}c_{4(g-j)}][b_{4(g-j)-1}c_{4(g-j)-3}]\right)+[B]\\
&&
+\left(\sum_{j=1}^{2g}c_{j+2}^*b_j^*[b_jc_{j+2}]+c_{j}^*b_{j+2}^*[b_{j+2}c_j]\right),
\end{eqnarray*}
the QP $(Q(\sigma),S+W')$ is a reduced part of $(\widetilde{\mu}_{2g}\ldots\widetilde{\mu}_1(\Qtau),\widetilde{\mu}_{2g}\ldots\widetilde{\mu}_1(\Ttau+W))$ and hence is (right-equivalent to) the mutation ${\mu}_{2g}\ldots{\mu}_1(\Qtau,\Ttau+W)$. Furthermore, from the fact that no arrow of the form $[b_jc_\ell]$ with $j\neq \ell$ appears in any of the terms of $\overline{S}$ we deduce that the coefficient in $S$ of any of the rotations of the cycle
$$
\left(\prod_{j=0}^{g-1}[b_{4(g-j)}c_{4(g-j)-2}][b_{4(g-j)-3}c_{4(g-j)-1}][b_{4(g-j)-2}c_{4(g-j)}][b_{4(g-j)-1}c_{4(g-j)-3}]\right)
$$
is $0$. Therefore, the coefficient of this cycle in $S+W'$ is $z$ (and its proper rotations do not appear).

The non-degeneracy of $(\Qtau,\Ttau+W)$ implies the non-degeneracy of $(Q(\sigma),S+W')$. Furthermore, it is easy to see that if we set $I=\{2g+1,2g+2,\ldots,6g-1,6g\}$, then $(Q(\sigma),S+W')$ and $I$ satisfy the hypotheses of \cite[Proposition 2.4]{GLFS}, from which we deduce that $z\neq 0$. This finishes the proof of Proposition \ref{prop:g-cycles-always-appear-2-punctures}.
\end{proof}

\begin{ex}\label{ex torus}
Figure \ref{fig:Flips toro dos ponchaduras} sketches the flip sequence in the proof of Proposition \ref{prop:g-cycles-always-appear-2-punctures} in the case of a twice-punctured torus. Note that the first and last triangulations have the same shape.

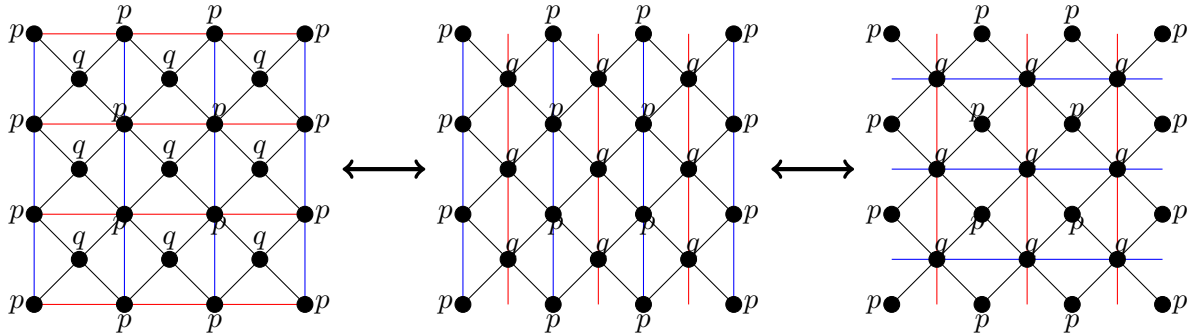
\begin{figure}[H]
\begin{center}
\begin{tikzpicture}[y=.1cm, x=.1cm,font=\normalsize]

\draw [red] (-18,18) to (18,18);
\draw [red] (-18,6) to (18,6);
\draw [red] (-18,-6) to (18,-6);
\draw [red] (-18,-18) to (18,-18);

\draw [blue] (-18,-18) to (-18,18);
\draw [blue] (-6,-18) to (-6,18);
\draw [blue] (6,-18) to (6,18);
\draw [blue] (18,-18) to (18,18);

\filldraw[fill=black!100,draw=black!100] (0,0) circle (3pt)    node[anchor=north] {};
\filldraw[fill=black!100,draw=black!100] (12,0) circle (3pt)    node[anchor=north] {};
\filldraw[fill=black!100,draw=black!100] (-12,0) circle (3pt)    node[anchor=north] {};
\filldraw[fill=black!100,draw=black!100] (0,12) circle (3pt)    node[anchor=north] {};
\filldraw[fill=black!100,draw=black!100] (0,-12) circle (3pt)    node[anchor=north] {};
\filldraw[fill=black!100,draw=black!100] (-12,-12) circle (3pt)    node[anchor=north] {};
\filldraw[fill=black!100,draw=black!100] (-12,12) circle (3pt)    node[anchor=north] {};
\filldraw[fill=black!100,draw=black!100] (12,12) circle (3pt)    node[anchor=north] {};
\filldraw[fill=black!100,draw=black!100] (12,-12) circle (3pt)    node[anchor=north] {};

\filldraw[fill=black!100,draw=black!100] (-18,-18) circle (3pt)    node[anchor=north] {};
\filldraw[fill=black!100,draw=black!100] (-18,-6) circle (3pt)    node[anchor=north] {};
\filldraw[fill=black!100,draw=black!100] (-18,6) circle (3pt)    node[anchor=north] {};
\filldraw[fill=black!100,draw=black!100] (-18,18) circle (3pt)    node[anchor=north] {};
\filldraw[fill=black!100,draw=black!100] (-6,-18) circle (3pt)    node[anchor=north] {};
\filldraw[fill=black!100,draw=black!100] (-6,-6) circle (3pt)    node[anchor=north] {};
\filldraw[fill=black!100,draw=black!100] (-6,6) circle (3pt)    node[anchor=north] {};
\filldraw[fill=black!100,draw=black!100] (-6,18) circle (3pt)    node[anchor=north] {};
\filldraw[fill=black!100,draw=black!100] (6,-18) circle (3pt)    node[anchor=north] {};
\filldraw[fill=black!100,draw=black!100] (6,-6) circle (3pt)    node[anchor=north] {};
\filldraw[fill=black!100,draw=black!100] (6,6) circle (3pt)    node[anchor=north] {};
\filldraw[fill=black!100,draw=black!100] (6,18) circle (3pt)    node[anchor=north] {};
\filldraw[fill=black!100,draw=black!100] (18,-18) circle (3pt)    node[anchor=north] {};
\filldraw[fill=black!100,draw=black!100] (18,-6) circle (3pt)    node[anchor=north] {};
\filldraw[fill=black!100,draw=black!100] (18,6) circle (3pt)    node[anchor=north] {};
\filldraw[fill=black!100,draw=black!100] (18,18) circle (3pt)    node[anchor=north] {};

\node [above] at (0,0) {$q$};
\node [above] at (12,0) {$q$};
\node [above] at (-12,0) {$q$};
\node [above] at (0,12) {$q$};
\node [above] at (0,-12) {$q$};
\node [above] at (12,12) {$q$};
\node [above] at (12,-12) {$q$};
\node [above] at (-12,-12) {$q$};
\node [above] at (-12,12) {$q$};

\node [left] at (-18,-18) {$p$};
\node [left] at (-18,-6) {$p$};
\node [left] at (-18,6) {$p$};
\node [left] at (-18,18) {$p$};
\node [right] at (18,-18) {$p$};
\node [right] at (18,-6) {$p$};
\node [right] at (18,6) {$p$};
\node [right] at (18,18) {$p$};
\node [above] at (-6,18) {$p$};
\node [above] at (6,18) {$p$};
\node [below] at (-6,-18) {$p$};
\node [below] at (6,-18) {$p$};
\node at (6.6,7.5) {$p$};
\node at (6.6,-7.5) {$p$};
\node at (-6.6,7.5) {$p$};
\node at (-6.6,-7.5) {$p$};

\draw (-18,18) to (18,-18);
\draw (-6,18) to (18,-6);
\draw (6,18) to (18,6);
\draw (-18,6) to (6,-18);
\draw (-18,-6) to (-6,-18);

\draw (-18,-18) to (18,18);
\draw (-18,-6) to (6,18);
\draw (-18,6) to (-6,18);
\draw (-6,-18) to (18,6);
\draw (6,-18) to (18,-6);

\draw [red] (57,-18) to (57,18);
\draw [red] (69,-18) to (69,18);
\draw [red] (45,-18) to (45,18);

\draw [blue] (39,-18) to (39,18);
\draw [blue] (51,-18) to (51,18);
\draw [blue] (63,-18) to (63,18);
\draw [blue] (75,-18) to (75,18);

\filldraw[fill=black!100,draw=black!100] (57,0) circle (3pt)    node[anchor=north] {};
\filldraw[fill=black!100,draw=black!100] (69,0) circle (3pt)    node[anchor=north] {};
\filldraw[fill=black!100,draw=black!100] (45,0) circle (3pt)    node[anchor=north] {};
\filldraw[fill=black!100,draw=black!100] (57,12) circle (3pt)    node[anchor=north] {};
\filldraw[fill=black!100,draw=black!100] (57,-12) circle (3pt)    node[anchor=north] {};
\filldraw[fill=black!100,draw=black!100] (45,-12) circle (3pt)    node[anchor=north] {};
\filldraw[fill=black!100,draw=black!100] (45,12) circle (3pt)    node[anchor=north] {};
\filldraw[fill=black!100,draw=black!100] (69,12) circle (3pt)    node[anchor=north] {};
\filldraw[fill=black!100,draw=black!100] (69,-12) circle (3pt)    node[anchor=north] {};

\filldraw[fill=black!100,draw=black!100] (39,-18) circle (3pt)    node[anchor=north] {};
\filldraw[fill=black!100,draw=black!100] (39,-6) circle (3pt)    node[anchor=north] {};
\filldraw[fill=black!100,draw=black!100] (39,6) circle (3pt)    node[anchor=north] {};
\filldraw[fill=black!100,draw=black!100] (39,18) circle (3pt)    node[anchor=north] {};
\filldraw[fill=black!100,draw=black!100] (51,-18) circle (3pt)    node[anchor=north] {};
\filldraw[fill=black!100,draw=black!100] (51,-6) circle (3pt)    node[anchor=north] {};
\filldraw[fill=black!100,draw=black!100] (51,6) circle (3pt)    node[anchor=north] {};
\filldraw[fill=black!100,draw=black!100] (51,18) circle (3pt)    node[anchor=north] {};
\filldraw[fill=black!100,draw=black!100] (63,-18) circle (3pt)    node[anchor=north] {};
\filldraw[fill=black!100,draw=black!100] (63,-6) circle (3pt)    node[anchor=north] {};
\filldraw[fill=black!100,draw=black!100] (63,6) circle (3pt)    node[anchor=north] {};
\filldraw[fill=black!100,draw=black!100] (63,18) circle (3pt)    node[anchor=north] {};
\filldraw[fill=black!100,draw=black!100] (75,-18) circle (3pt)    node[anchor=north] {};
\filldraw[fill=black!100,draw=black!100] (75,-6) circle (3pt)    node[anchor=north] {};
\filldraw[fill=black!100,draw=black!100] (75,6) circle (3pt)    node[anchor=north] {};
\filldraw[fill=black!100,draw=black!100] (75,18) circle (3pt)    node[anchor=north] {};

\node at (57.6,1.5) {$q$};
\node at (69.6,13.5) {$q$};
\node at (57.6,13.5) {$q$};
\node at (69.6,1.5) {$q$};
\node at (57.6,-10.5) {$q$};
\node at (69.6,-10.5) {$q$};
\node at (45.6,1.5) {$q$};
\node at (45.6,13.5) {$q$};
\node at (45.6,-10.5) {$q$};

\node [left] at (39,-18) {$p$};
\node [left] at (39,-6) {$p$};
\node [left] at (39,6) {$p$};
\node [left] at (39,18) {$p$};
\node [right] at (75,-18) {$p$};
\node [right] at (75,-6) {$p$};
\node [right] at (75,6) {$p$};
\node [right] at (75,18) {$p$};
\node [above] at (51,18) {$p$};
\node [above] at (63,18) {$p$};
\node [below] at (51,-18) {$p$};
\node [below] at (63,-18) {$p$};
\node at (63.6,7.5) {$p$};
\node at (63.6,-7.5) {$p$};
\node at (51.4,7.5) {$p$};
\node at (51.4,-7.5) {$p$};

\draw (39,18) to (75,-18);
\draw (51,18) to (75,-6);
\draw (63,18) to (75,6);
\draw (39,6) to (63,-18);
\draw (39,-6) to (51,-18);

\draw (39,-18) to (75,18);
\draw (39,-6) to (63,18);
\draw (39,6) to (51,18);
\draw (51,-18) to (75,6);
\draw (63,-18) to (75,-6);

\draw [red] (114,-18) to (114,18);
\draw [red] (126,-18) to (126,18);
\draw [red] (102,-18) to (102,18);

\draw [blue] (96,12) to (132,12);
\draw [blue] (96,-12) to (132,-12);
\draw [blue] (96,0) to (132,0);

\filldraw[fill=black!100,draw=black!100] (114,0) circle (3pt)    node[anchor=north] {};
\filldraw[fill=black!100,draw=black!100] (126,0) circle (3pt)    node[anchor=north] {};
\filldraw[fill=black!100,draw=black!100] (102,0) circle (3pt)    node[anchor=north] {};
\filldraw[fill=black!100,draw=black!100] (114,12) circle (3pt)    node[anchor=north] {};
\filldraw[fill=black!100,draw=black!100] (114,-12) circle (3pt)    node[anchor=north] {};
\filldraw[fill=black!100,draw=black!100] (102,-12) circle (3pt)    node[anchor=north] {};
\filldraw[fill=black!100,draw=black!100] (102,12) circle (3pt)    node[anchor=north] {};
\filldraw[fill=black!100,draw=black!100] (126,12) circle (3pt)    node[anchor=north] {};
\filldraw[fill=black!100,draw=black!100] (126,-12) circle (3pt)    node[anchor=north] {};

\filldraw[fill=black!100,draw=black!100] (96,-18) circle (3pt)    node[anchor=north] {};
\filldraw[fill=black!100,draw=black!100] (96,-6) circle (3pt)    node[anchor=north] {};
\filldraw[fill=black!100,draw=black!100] (96,6) circle (3pt)    node[anchor=north] {};
\filldraw[fill=black!100,draw=black!100] (96,18) circle (3pt)    node[anchor=north] {};
\filldraw[fill=black!100,draw=black!100] (108,-18) circle (3pt)    node[anchor=north] {};
\filldraw[fill=black!100,draw=black!100] (108,-6) circle (3pt)    node[anchor=north] {};
\filldraw[fill=black!100,draw=black!100] (108,6) circle (3pt)    node[anchor=north] {};
\filldraw[fill=black!100,draw=black!100] (108,18) circle (3pt)    node[anchor=north] {};
\filldraw[fill=black!100,draw=black!100] (120,-18) circle (3pt)    node[anchor=north] {};
\filldraw[fill=black!100,draw=black!100] (120,-6) circle (3pt)    node[anchor=north] {};
\filldraw[fill=black!100,draw=black!100] (120,6) circle (3pt)    node[anchor=north] {};
\filldraw[fill=black!100,draw=black!100] (120,18) circle (3pt)    node[anchor=north] {};
\filldraw[fill=black!100,draw=black!100] (132,-18) circle (3pt)    node[anchor=north] {};
\filldraw[fill=black!100,draw=black!100] (132,-6) circle (3pt)    node[anchor=north] {};
\filldraw[fill=black!100,draw=black!100] (132,6) circle (3pt)    node[anchor=north] {};
\filldraw[fill=black!100,draw=black!100] (132,18) circle (3pt)    node[anchor=north] {};

\node at (114.6,1.5) {$q$};
\node at (126.6,13.5) {$q$};
\node at (114.6,13.5) {$q$};
\node at (126.6,1.5) {$q$};
\node at (114.6,-10.5) {$q$};
\node at (126.6,-10.5) {$q$};
\node at (102.6,1.5) {$q$};
\node at (102.6,13.5) {$q$};
\node at (102.6,-10.5) {$q$};

\node [left] at (96,-18) {$p$};
\node [left] at (96,-6) {$p$};
\node [left] at (96,6) {$p$};
\node [left] at (96,18) {$p$};
\node [right] at (132,-18) {$p$};
\node [right] at (132,-6) {$p$};
\node [right] at (132,6) {$p$};
\node [right] at (132,18) {$p$};
\node [above] at (108,18) {$p$};
\node [above] at (120,18) {$p$};
\node [below] at (108,-18) {$p$};
\node [below] at (120,-18) {$p$};
\node at (120.6,7.5) {$p$};
\node at (120.6,-7.5) {$p$};
\node at (107.4,7.5) {$p$};
\node at (107.4,-7.5) {$p$};

\draw (96,18) to (132,-18);
\draw (108,18) to (132,-6);
\draw (120,18) to (132,6);
\draw (96,6) to (120,-18);
\draw (96,-6) to (108,-18);

\draw (96,-18) to (132,18);
\draw (96,-6) to (120,18);
\draw (96,6) to (108,18);
\draw (108,-18) to (132,6);
\draw (120,-18) to (132,-6);

\draw [<->, ultra thick] (23,0) to (34,0);
\draw [<->, ultra thick] (80,0) to (91,0);

\end{tikzpicture}\\
\caption{Proving Proposition \ref{prop:g-cycles-always-appear-2-punctures} for the twice-punctured torus.}
\label{fig:Flips toro dos ponchaduras}
\end{center}
\end{figure}
\end{ex}

\begin{proof}[Proof of Theorem \ref{thm:uniqueness-two-punctures}] Let $\surf$ a be twice-punctured closed surface of positive genus, and let $\tau$ be a triangulation of $\surf$ satisfying \eqref{eq:valency-at-least-4} and \eqref{eq:no-double-arrows}. By Lemma \ref{lemma:non-degenerate-potentials-contains-Ttau}, every non-degenerate potential on $\Qtau$ is right-equivalent to a potential of the form $\Ttau+U$ for some $U$ which is rotationally disjoint from $\Ttau$. By Proposition \ref{prop:replacing-with-sums-of-powers-of-g-cycles}, $\Ttau+U$ is right-equivalent to $\Ttau+W$ for some potential that involves only positive powers of $g$-cycles. Theorem \ref{thm:uniqueness-two-punctures} now follows from Proposition \ref{prop:g-cycles-always-appear-2-punctures}, Lemma \ref{lemma:last-cancellation--g-powers} and \cite[Lemma 8.5]{GLFS}.
\end{proof}
  \section*{Acknowledgements}

We thank Christof Geiss and Jan Schr\"{o}er for many helpful discussions.

The three authors were supported by the second author's grant PAPIIT-IA102215. The first two authors were supported by the second author's grant CONACyT-238754 as well. DLF received support from a \emph{C\'{a}tedra Marcos Moshinsky} and the grant PAPIIT-IN112519.


\begin{thebibliography}{50}



\bibitem{Bridgeland-Smith} Tom Bridgeland, Ivan Smith. {\it Quadratic differentials as stability conditions}. Publications mathématiques de l'IHÉS, June 2015, Volume 121, Issue 1, 155--278. arXiv:1302.7030



\bibitem{DWZ1} Harm Derksen, Jerzy Weyman, Andrei Zelevinsky. {\it Quivers with potentials and their representations I: Mutations}. Selecta Math. 14 (2008), no. 1, 59--119. arXiv:0704.0649



\bibitem{FST} Sergey Fomin, Michael Shapiro, Dylan Thurston. {\it Cluster algebras and triangulated surfaces, part I: Cluster complexes}. Acta Mathematica 201 (2008), 83-146. arXiv:math.RA/0608367



\bibitem{GLFS} Christof Geiss, D. Labardini-Fragoso, Jan Schr\"oer. {\it The representation type of Jacobian algebras}. Advances in Mathematics, Vol. 290 (2016), 364-452. doi:10.1016/j.aim.2015.09.038 \ \ arXiv:1308.0478



\bibitem{Geuenich-master-thesis} J. Geuenich. {\it Quivers with potentials on three vertices}. Master thesis. Universit\"at Bonn. 2013.



\bibitem{Keller-Yang} Bernhard Keller, Dong Yang. {\it Derived equivalences from mutations of quiver with potential}. Adv. Math.
226: 2118--2168, (2011).

\bibitem{Labardini1} D. Labardini-Fragoso. {\it Quivers with potentials associated to triangulated surfaces}. Proc. London Mathematical Society (2009) 98 (3): 797-839. arXiv:0803.1328

\bibitem{Labardini2} Daniel Labardini-Fragoso. {\it Quivers with potentials associated to triangulated surfaces, part II: Arc representations}. arXiv:0909.4100

\bibitem{Labardini4} D. Labardini-Fragoso. {\it Quivers with potentials associated to triangulated surfaces, part IV: Removing boundary assumptions}. Selecta Mathematica (New series), Vol. 22 (2016), Issue 1, 145--189. DOI: 10.1007/s00029-015-0188-8  arXiv:1206.1798

\bibitem{Labardini-survey} D. Labardini-Fragoso. {\it On triangulations, quivers with potentials and mutations}. Contemporary Mathematics (AMS), Vol. 657 ``Mexican Mathematicians Abroad: Recent Contributions'' (2016), 103--127. http://dx.doi.org/10.1090/conm/657/13092 \ \ arXiv:1302.1936



\bibitem{Ladkani} Sefi Ladkani. {\it On Jacobian algebras from closed surfaces}. arXiv:1207.3778

\bibitem{Ladkani2} Sefi Ladkani. {\it On cluster algebras from once punctured closed surfaces}. arXiv:1310.4454

\bibitem{Miranda-undergrad-thesis} J. L. Miranda-Olvera. {\it Carcajes con potencial no degenerados asociados a triangulaciones de superficies: Existencia y unicidad}. Undergraduate thesis. Facultad de Ciencias, Universidad Nacional Aut\'onoma de M\'exico. 2016.
    
\bibitem{Mosher} Lee Mosher. {\it Tiling the projective foliation space of a punctured surface}. Trans. Amer. Math. Soc. 306 (1988), 1--70.




\bibitem{Smith} Ivan Smith. {\it Quiver algebras as Fukaya categories}. Geom. Topol. 19 (2015), no. 5, 2557–2617. arXiv:1309.0452



\end{thebibliography}
\end{document}